%% file: ms.tex
\documentclass{jfp1-CHANGED}

\expandafter\let\csname equation*\endcsname\relax
\expandafter\let\csname endequation*\endcsname\relax

\usepackage{amsmath}
\usepackage{amsfonts}
\usepackage{enumitem}
\usepackage{tikz-cd}
\usepackage{stmaryrd}
\usepackage{bussproofs}
\usepackage{natbib}

\usepackage{color}
\def\bR{\begin{color}{red}}
\def\bB{\begin{color}{blue}}
\def\bM{\begin{color}{magenta}}
\def\bC{\begin{color}{cyan}}
\def\bW{\begin{color}{white}}
\def\bBl{\begin{color}{black}}
\def\bG{\begin{color}{green}}
\def\bY{\begin{color}{yellow}}
\def\e{\end{color}}

\newtheorem{theorem}{Theorem}
\newtheorem{lemma}[theorem]{Lemma}
\newtheorem{corollary}[theorem]{Corollary}

\newtheorem{definition}[theorem]{Definition}
\newtheorem{example}[theorem]{Example}
\newtheorem{counter}[theorem]{Counterexample}
\newtheorem{remark}[theorem]{Remark}
\numberwithin{theorem}{section}

\newcommand{\define}[1]{\textbf{#1}}
\newcommand{\reason}[1]{\;\;\left\{\; #1 \;\right\}}
\newcommand{\bb}[1]{\mathbb{#1}}
\newcommand{\given}{\, | \,}
\newcommand{\concat}{\mathbin{+\mkern-10mu+}}

\newcommand{\theoryeq}[1]{=_\mathbb{#1}}
\newcommand{\specialopS}{\mathbf{s}}
\newcommand{\specialopT}{\mathbf{t}}
\newcommand{\anyopS}{\phi}

\newcommand{\signatureS}{\Sigma^\bb{S}}
\newcommand{\signatureT}{\Sigma^\bb{T}}
\newcommand{\eqT}{E^\bb{T}}
\newcommand{\essuniqfunctionA}{f}
\newcommand{\essuniqfunctionB}{f'}
\mathchardef\mathhyphen="2D

\newcommand{\alge}[2]{(#1, #2)\mathhyphen\mathbf{Alg}}
\usepackage{xspace}
\newcommand{\catset}{\mathbf{Set}}

\newcommand{\refsec}{Section}

\newcommand{\refthm}{Theorem}
\newcommand{\refthms}{Theorems}
\newcommand{\reflem}{Lemma}

\newcommand{\refex}{Example}
\newcommand{\refexs}{Examples}
\newcommand{\refcounter}{Counterexample}

\newcommand{\reffig}{Figure}
\newcommand{\refcor}{Corollary}
\newcommand{\refeq}{Equation}
\newcommand{\refeqs}{Equations}
\newcommand{\refdef}{Definition}

\newcommand{\refpro}{property}

\newcommand{\reftab}{Table}
\newcommand{\refrem}{Remark}

\DeclareMathOperator{\var}{var}
\DeclareMathOperator{\supp}{supp}
\DeclareMathOperator{\const}{const}
\DeclareMathOperator{\fst}{head}

\newcommand{\plotkinspecial}{\Xi}
\newcommand{\treemonad}{B}

\title[No-Go Theorems for Distributive Laws]
      {Don't Try This at Home: \\No-Go Theorems for Distributive Laws}

 \author[M. Zwart and D. Marsden]
        {MAAIKE ZWART and DAN MARSDEN\footnote{Authors are ordered according to the outcome of a random number generator.}\\
         Department of Computer Science, University of Oxford\\
         \email{maaike.zwart@cs.ox.ac.uk, daniel.marsden@cs.ox.ac.uk}}

\begin{document}
\input{prelims}
\input{plotkin}

\input{zwart}
\input{comparisons}

\input{boom}

\input{end}

\bibliography{paper}
\bibliographystyle{jfp}
\end{document}

%% file: prelims.tex
\maketitle

\begin{abstract}
Beck's distributive laws provide sufficient conditions under which two monads can be composed, and monads arising from distributive laws have many desirable theoretical properties. Unfortunately, finding and verifying distributive laws, or establishing if one even exists, can be extremely difficult and error-prone.

We develop general-purpose techniques for showing when there can be no distributive law between two monads.
Two approaches are presented. The first widely generalizes ideas from a counterexample attributed to Plotkin, yielding general-purpose theorems that recover the previously known situations in which no distributive law can exist. Our second approach is entirely novel, encompassing new practical situations beyond our generalization of Plotkin's approach. It negatively resolves the open question of whether the list monad distributes over itself.

Our approach adopts an algebraic perspective throughout, exploiting a syntactic characterization of distributive laws. This approach is key to generalizing beyond what has been achieved by direct calculations in previous work. We show via examples many situations in which our theorems can be applied. This includes a detailed analysis of distributive laws for members of an extension of the Boom type hierarchy, well known to functional programmers.
\end{abstract}

\section{Introduction}
Monads have become a key tool in computer science. They are, amongst other things, used to provide semantics for computational effects such as state, exceptions and IO~\citep{Moggi1991}, and to structure functional programs~\citep{Wadler1995,PeytonJones2001}. They even appear explicitly in the standard library of the Haskell programming language~\citep{HaskellReport}.

Monads are a categorical concept. A monad on a category~$\mathcal{C}$ consists of an endofunctor~$T$ and two natural transformations~$1 \Rightarrow T$ and~$\mu: T \circ T \Rightarrow T$ satisfying axioms described in \refsec~\ref{sec:monads}. Given two monads with underlying functors~$S$ and~$T$, it is natural to ask if~$T \circ S$ always carries the structure of a monad. This would, for example, provide a modular way to model complex computational effects by combining simpler monads together. Unfortunately, in general monads cannot be combined in this way. Beck has shown that the existence of a~\emph{distributive law} provides sufficient conditions for~$T \circ S$ to form a monad \citep{Beck1969}. A distributive law is a natural transformation of type:
\begin{equation*}
  S \circ T \Rightarrow T \circ S
\end{equation*}
satisfying four equations described in \refsec~\ref{sec:distributive-laws}. This important idea has since been generalized to identify notions of a distributive law for combining monads with comonads, monads with pointed endofunctors, endofunctors with endofunctors and various other combinations.

General purpose techniques have been developed for constructing distributive laws~\citep{ManesMulry2007,ManesMulry2008, Bonsangue2013, Dahlqvist2017}. These general-purpose techniques are highly valuable, for as stated in~\citet{Bonsangue2013}: \emph{``It can be rather difficult to prove the defining axioms of a distributive law.''}. In fact, it can be so difficult that on occasion a distributive law has been published which later turned out to be incorrect; see \citet{KlinSalamanca2018} for an overview of such cases involving the powerset monad.

The literature has tended to focus on positive results, either demonstrating specific distributive laws, or developing general-purpose techniques for constructing them. By comparison, there is a relative paucity of negative results, showing when no distributive law can exist.
The most well-known result of this type appears in~\citet{VaraccaWinskel2006}, where it is shown that there is no distributive law combining the powerset and finite probability distribution monads, via a proof credited to Plotkin. This result was later strengthened by \citet{DahlqvistNeves2018}, who showed that the composite functor carries no monad structure at all. Recently, the same proof technique was used by~\citet{KlinSalamanca2018} to show that composing the covariant powerset functor  with itself yields an endofunctor that does not carry any monad structure, correcting an earlier error in the literature. To the best our knowledge, these are currently the only impossibility results in the literature.

In this paper we significantly extend the understanding of when no distributive law between two monads can exist. We restrict our attention to distributive laws between monads on the category of sets and functions, presenting several theorems for proving the non-existence of distributive laws for large classes of monads on this category. These results can roughly be divided into two classes:
\begin{itemize}
\item Firstly, we generalize Plotkin's method to a general-purpose theorem covering all the existing published no-go results about distributive laws, and yielding new results as well.
\item Secondly, we identify an entirely new approach, distinct from Plotkin's method. This novel technique leads to new general-purpose theorems, covering combinations of monads that were previously out of reach.
\end{itemize}
As one application of our second method, we show that the list monad cannot distribute over itself, resolving an open question in the literature~\citep{ManesMulry2007, ManesMulry2008}. In addition, the no-go theorems produced by the second method revealed yet another faulty distributive law in the literature, involving list and exception monads.

Monads have deep connections with universal algebra. We exploit this algebraic viewpoint by using an explicit description of the algebraic nature of distributive laws in~\citet{PirogStaton2017}, which was inspired by the work in \citet{Cheng2011}. Formulating our results in algebraic form is a key contribution of our work, simplifying and clarifying the essentials of our proofs, which can be obscured by more direct calculations.

In physics, theorems proving the impossibility of certain things are called \emph{no-go theorems}, because they clearly identify theoretical directions that cannot succeed. We follow this example, and hope that by sharing our results, we prevent others from wasting time on forlorn searches for distributive laws that cannot exist.

\subsection{Contributions}
We briefly outline our contribution.
\begin{itemize}
\item We demonstrate the non-existence of distributive laws for large classes of monads by taking an algebraic perspective, using a characterisation of distributive laws as described in~\citet{PirogStaton2017}:
\begin{itemize}
\item In \refthms~\ref{thm:plotkin-1} and~\ref{thm:plotkin-2} we widely generalize the essentials of a counterexample of Plotkin. Our general theorems reveal precisely which properties are key to the proof, hence capturing an infinite class of monads sharing these properties.
\item In \refsec~\ref{section:beyondPlotkin} we present three entirely new no-go theorems:
     \begin{itemize}
     \item Theorem~\ref{thm:nogoConstants} states when having more than one constant in a theory is problematic. The usefulness of this theorem is demonstrated by exhibiting a previously unidentified error in a theorem in the literature.
     \item Theorem~\ref{thm:nogoTreeList} is a no-go theorem for monads that do not satisfy the abides equation. This theorem negatively answers the open question of whether the list monad distributes over itself.
     \item Theorem~\ref{thm:nogoTreeIdem} is a no-go theorem focussing on the combination of idempotence and units. It proves there is no distributive law for the powerset monad over the multiset monad: $PM \Rightarrow MP$.
     \end{itemize}
\end{itemize}
\item Also in \refsec~\ref{section:beyondPlotkin}, \refthm~\ref{theorem} gives conditions under which at most one distributive law can exist.
\item Section~\ref{sec:comparison} compares the algebraic methods used in our proofs to the standard categorical approach which directly uses Beck's axioms for distributive laws.
\item Finally, we illustrate the scope of the no-go theorems in this paper by studying an extended version of the Boom hierarchy~\citep{Meertens1986}. We establish whether or not a distributive law exists for many combinations of monads within this hierarchy, greatly expanding what was previously known.
\end{itemize}
Throughout this paper we restrict our attention to monads on the category of sets and functions, as this is already an incredibly rich setting.

\section{Preliminaries}\label{sec:prelims}
We assume the reader has some familiarity with category theory, in particular the notions of category, functor, natural transformation and adjunction. We recommend~\citet{Awodey2010, BarrWells1990} and the standard reference~\citet{MacLane2013} as sources of background material in category theory. A good source for basic monad theory is~\citet{BarrWells1985}.

In this section, we review the definitions of monads and distributive laws, and briefly explain the connection between monads and universal algebra. We also introduce the algebraic notation we use in the rest of this paper.

As we will need several different types of compositions of functors and natural transformations, we first introduce some notation:
\begin{itemize}
\item Composition of functors $S$ and $T$ will be written as $T \circ S$.
\item Horizontal composition of natural transformations will also be written with $\circ$. In addition, compositions involving the identity natural transformation $1_T: T \Rightarrow T$ and natural transformation $\alpha$ will be written as $\alpha \circ T$ and $T \circ \alpha$ rather than $\alpha \circ 1_T$ and $1_T \circ \alpha$.
\item Vertical composition of two natural transformations $\alpha$ and $\alpha'$ will be written as ${\alpha' \cdot \alpha}$.
\end{itemize}
We will leave the symbols $\circ, \cdot$ out whenever this can be done without causing confusion.

\subsection{Monads}
\label{sec:monads}
We introduce monads in what is sometimes referred to as monoid form, see~\cite{Manes1976}, this is the standard approach used by category theorists, and is most convenient for later developments. This definition is equivalent to the definition commonly used by functional programmers, the connection between these two formulations can be found in~\citet{Manes1976}.

\begin{definition}[Monad]
For any category $\mathcal{C}$, a \define{monad} $\langle T,\eta, \mu \rangle$ on $\mathcal{C}$ consists of:
 \begin{itemize}
   \item An endofunctor $T: \mathcal{C} \rightarrow \mathcal{C}$.
   \item A natural transformation $\eta: 1 \Rightarrow T$ called the \emph{unit}.
   \item A natural transformation $\mu: TT \Rightarrow T$ called the \emph{multiplication}.
 \end{itemize}
 such that the following axioms hold:
\begin{align}
  \mu \cdot T\eta & = id \tag{unit 1}\\
  \mu \cdot \eta T & = id \tag{unit 2}\\
  \mu \cdot T\mu & = \mu \cdot \mu T\tag{associativity}
\end{align}
Or, as commuting diagrams:
\begin{center}
\begin{tikzcd}
TX \arrow[r, "\eta T"]\arrow[d, "T\eta"'] \arrow[rd, equal] & TTX \arrow[d, "\mu"] & & TTTX \arrow[r, "\mu T"]\arrow[d, "T \mu"'] & TTX \arrow[d, "\mu"] \\
TTX \arrow[r, "\mu"] & TX & & TTX \arrow[r, "\mu"] & TX
\end{tikzcd}
\end{center}
\end{definition}

All the monads in this paper will be monads on the category $\catset$ of sets and functions. We give a few examples, all of which will play important roles later on.

\newpage

\begin{example}[The Exception and Lift Monads]
  \label{ex:exception-monad}
  For any set $E$, the exception monad $\langle (-+E),\eta^E, \mu^E \rangle$ is given by:
  \begin{itemize}
  \item $(-+E): \catset \rightarrow \catset$ maps a set $X$ to the disjoint union $X+E$ of $X$ and $E$.
  \item $\eta^E_X: X \rightarrow X+E$ is the left inclusion morphism $i_l: X \rightarrow X+E$.
  \item $\mu^E_X: ((X+E)+E) \rightarrow (X+E)$ is the identity on $X$, and collapses the two copies $E$ down to a single copy.
  \end{itemize}
  When $E$ is a singleton set, this monad is also known as the maybe monad or the lift monad, written as $(-)_\bot$.
\end{example}
\begin{example}[The List Monad]
  \label{ex:list-monad}
  The list monad $\langle L,\eta^L, \mu^L \rangle$ is given by:
  \begin{itemize}
  \item $L: \catset \rightarrow \catset$ maps a set $X$ to the set containing all finite lists with elements taken from $X$.
  \item $\eta^L_X: X \rightarrow LX$ maps an element $x \in X$ to the singleton list $[x]$.
  \item $\mu^L_X: LLX \rightarrow LX$ concatenates a list of lists to form a single list.
    \[
    \mu^L_X ([[a,b],[c,d,e],[f]]) = [a,b]\concat[c,d,e]\concat[f] = [a,b,c,d,e,f]
    \]
  \end{itemize}
  For reasons that will become clear later on, this monad is also known as the monoid monad.
\end{example}
\begin{example}[The Powerset Monad]
  \label{ex:powerset-monad}
  The finite covariant powerset monad $\langle P,\eta^P, \mu^P \rangle$ is given by:
  \begin{itemize}
  \item $P: \catset \rightarrow \catset$ maps a set $X$ to the set containing all finite subsets of $X$.
  \item $\eta^P_X: X \rightarrow PX$ maps an element $x \in X$ to the singleton set $\{x\}$.
  \item $\mu^P_X: PPX \rightarrow PX$ takes a union of sets.
    \[
    \mu^P_X \{\{a,b,c\},\{b,c\},\{c,d,e\}\} = \{a,b,c\} \cup \{b,c\} \cup \{c,d,e\} = \{a,b,c,d,e\}
    \]
  \end{itemize}
\end{example}
\begin{example}[The Binary Tree Monad]
  \label{ex:tree-monad}
  The binary tree monad $\langle \treemonad,\eta^{\treemonad}, \mu^{\treemonad} \rangle$ is given by:
  \begin{itemize}
  \item $\treemonad: \catset \rightarrow \catset$ maps a set $X$ to the set containing all binary trees with leaves labelled by elements from $X$.
  \item $\eta^{\treemonad}_X: X \rightarrow \treemonad X$ maps an element $x \in X$ to the tree consisting of a single leaf labelled with $x$.
  \item $\mu^{\treemonad}_X: \treemonad\treemonad X \rightarrow \treemonad X$ flattens a tree of trees.
  \end{itemize}
\end{example}
\begin{example}[The Reader Monad]
  \label{ex:reader-monad}
  For any set of states~$R$, the reader monad~$\langle (-)^R,\eta^{R}, \mu^{R} \rangle$ is given by:
  \begin{itemize}
  \item $(-)^R: \catset \rightarrow \catset$ maps a set $X$ to the set of functions $X^R$ from $R$ to $X$.
  \item $\eta^{R}_X: X \rightarrow X^R$ maps an element $x \in X$ to the constant function $\const_x$.
  \item $\mu^{R}_X: (X^R)^R \rightarrow X^R$ works as follows: let $f: R \rightarrow X^R$ be an element of $(X^R)^R$. Then for any $r \in R$, $f(r)$ is a function from $R$ to $X$. So:
    \[
    \mu^R_X(f) = \lambda r \rightarrow (f(r))(r)
    \]
  \end{itemize}
\end{example}
For our final two examples, the following notion will be useful.
\begin{definition}
\label{def:finite-support}
A function with codomain the natural numbers or positive reals is said to have~\define{finite support} if only finitely many elements of~$X$ are mapped to non-zero values.
\end{definition}
\begin{example}[The Multiset and Abelian Group Monads]
  \label{ex:multiset-monad}
  A~\define{multiset} with elements in set~$X$ is a function from~$X$ to the natural numbers, mapping each element to its multiplicity. A~\define{finite multiset} is a multiset with finite support. As we will only be interested in finite multisets, we will abuse terminology and simply refer to them as multisets from this point. Multisets will be described using a set theoretic notation, for example the function:
  \begin{equation*}
    a \mapsto 1 \quad b \mapsto 2 \quad c \mapsto 0
  \end{equation*}
  will be written as:
  \begin{equation*}
    \{ a:1, b:2 \}
  \end{equation*}
  Elements with zero multiplicity may be omitted when the underlying set is clear from the context.

  The multiset monad~$\langle M,\eta^M, \mu^M \rangle$ is given by:
  \begin{itemize}
  \item $M: \catset \rightarrow \catset$ maps a set $X$ to the set containing all finite multisets with elements taken from $X$.
  \item $\eta^M_X: X \rightarrow MX$ maps an element $x \in X$ to the singleton $\{x:1\}$.
  \item $\mu^M_X: MMX \rightarrow MX$ takes a union of multisets, carefully accounting for all multiplicities.
    \[
    \mu^M_X \{\{a:1, b:2\}:1, \{b:1\}:3\} = \{a:(1 \cdot 1), b:(1\cdot2 + 3\cdot1)\} = \{a:1, b:5\}
    \]
    We can generalize the notion of multiset to take multiplicities in the integers. This results in another monad, the~\emph{Abelian group monad}.
  \end{itemize}
\end{example}
\begin{example}[The Distribution Monad]
  \label{ex:distribution-monad}
  The finite probability distribution monad $\langle D,\eta^{D}, \mu^{D} \rangle$ is given by:
  \begin{itemize}
  \item $D: \catset \rightarrow \catset$ maps a set $X$ to the set containing all finite probability distributions over $X$. That is, the set of all functions $d: X \rightarrow [0,1]$ which have finite support and:
       \[\sum_{x \in \supp(d)} d(x) = 1\]
  \item $\eta^{D}_X: X \rightarrow DX$ maps an element $x \in X$ to the distribution $\delta_x$, where:
    \[
    \delta_x (y) = \begin{cases}
      1, & \mbox{if } y = x \\
      0, & \mbox{otherwise}.
    \end{cases}
    \]
  \item $\mu^{D}: DDX \rightarrow DX$ takes a weighted average. If $e: DX \rightarrow [0,1]$ is in $DDX$, then:
    \[
    \mu^D_X(e) = \lambda x \rightarrow \sum_{d \in \supp(e)} e(d)d(x)
    \]
  \end{itemize}
\end{example}

\newpage
\subsection{Distributive Laws}
\label{sec:distributive-laws}
Distributive laws are particular natural transformations, introduced by~\citet{Beck1969}, to facilitate monad compositions. We begin with the formal definition.

\begin{definition}[Distributive Law]\label{def:distlaw}
Given monads $S$ and $T$, a \define{distributive law} for monad composition $TS$ is a natural transformation $\lambda: ST \Rightarrow TS$ satisfying the following axioms:
\begin{align}
  \lambda \cdot \eta^S T & = T\eta^S \tag{unit1}\\
  \lambda \cdot S\eta^T  & = \eta^T S \tag{unit2} \\
  \lambda \cdot \mu^S T & = T\mu^S \cdot \lambda S \cdot S\lambda \tag{multiplication1} \\
  \lambda \cdot S\mu^T  & = \mu^T S \cdot T\lambda \cdot \lambda T \tag{multiplication2}
\end{align}
Or, as commuting diagrams:
\begin{center}
\begin{tikzcd}\label{monad_squared_eq}
& T \arrow[dl, "\eta^S T"']\arrow[dr, "T \eta^S"] & & SST \arrow[d, "\mu^S T"'] \arrow[r, "S \lambda"] & STS \arrow[r, "\lambda S"] & TSS \arrow[d, "T \mu^S"] \\
ST \arrow[rr, "\lambda"]& & TS & ST \arrow[rr,"\lambda"] & & TS \\
& S \arrow[dl, "S\eta^T"'] \arrow[dr, "\eta^T S"] & & STT \arrow[d, "S \mu^T"'] \arrow[r, "\lambda T"] & TST \arrow[r, "T\lambda"] & TTS \arrow[d, "\mu^T S"] \\
 ST \arrow[rr, "\lambda"] & & TS & ST \arrow[rr,"\lambda"] & & TS \\
\end{tikzcd}
\end{center}
\end{definition}
The axioms of distributive laws may appear somewhat cryptic on first examination. In fact they are entirely natural when viewed in the graphical language of string diagrams; the interested reader may wish to consult~\citet{HinzeMarsden2016} for further details.

\begin{remark}
For a pair of monads~$S, T$ the expression~``$T$ distributes over~$S$'' is often used. This phrasing is somewhat ambiguous and prone to errors. It could mean the underlying natural transformation has type~$TS \Rightarrow ST$ or~$ST \Rightarrow TS$, and conventions are applied inconsistently in the literature. We will therefore avoid this approach, and explicitly state the type of the natural transformation, for example~``there is a distributive law of type~$TS \Rightarrow ST$''.
\end{remark}
\begin{theorem}[\citet{Beck1969}]
  Let~$\mathcal{C}$ be a category, and~$\langle S, \eta^S, \mu^S \rangle$ and~$\langle T, \eta^T, \mu^T \rangle$ two monads on~$\mathcal{C}$. If~$\lambda: ST \Rightarrow TS$ is a distributive law, then~$TS$ carries a monad structure with:
  \begin{itemize}
  \item Unit given by~$\eta^T\eta^S$.
  \item Multiplication given by~$\mu^T\mu^S \cdot T\lambda S$
  \end{itemize}
\end{theorem}

\begin{example}[The Ring Monad]
  \label{ex:ring-monad}
  The most famous example of a distributive law involves the list monad and the Abelian group monad, $\lambda: LA \Rightarrow  AL$~\citep{Beck1969}. It captures exactly the distributivity of multiplication over addition. Writing a list as a formal product and a multiset with integer values as a formal sum, the distributive law can be written as:
  \[
  \lambda \left( \prod_{i=0}^{n} \sum_{j_i = 0}^{m_i} x_{ij_i} \right) = \sum_{j_0 = 0}^{m_0} \cdots \sum_{j_n}^{m_n} \prod_{i=0}^{n} x_{ij_i}
  \]
  or more simply:
  \begin{equation}
  \label{timesoverplus}
  \lambda \left( a \cdot (b + c) \right) = (a \cdot b) + (a \cdot c)
  \end{equation}
  The resulting composite monad is the ring monad, which like the Abelian group monad is named after its algebras rather than its monad action.

  The term distributive law is motivated by the distributivity of multiplication over addition seen in this example, and in fact many other distributive laws exploit similar algebraic properties. Caution is needed though: the validity of an equation such as~\eqref{timesoverplus} does not automatically imply the existence of a distributive law.
\end{example}

\begin{example}[Multiset Monad Composed with Itself]\label{ex:multisetdistlaw}
  The multiset monad distributes over itself. The distributive law is best described as taking a Cartesian product. For a multiset of multisets $\mathcal{M}$, $\lambda(\mathcal{M})$ is computed as follows: take the Cartesian product of all the multisets in~$\mathcal{M}$. Then, for each tuple of this Cartesian product, gather the elements of the tuple into a multiset. The multiset of all thus created multisets is the result of $\lambda(\mathcal{M})$. So for example:
  \begin{align*}
  \lambda \left( \{ \{a:1, b:2\}:2\}\right) & = \lambda \left( \{ \{a,b,b\},\{a,b,b\}\}\right) \\
  & = \{\{a,a\},\{a,b\},\{a,b\},\\
  & \quad\;\;\;\,\{b,a\},\{b,b\},\{b,b\},\\
  & \quad\;\;\;\,\{b,a\},\{b,b\},\{b,b\}\}\\
  & = \{\{a:2\}:1,\{a:1,b:1\}:4, \{b:2\}:4 \}
  \end{align*}
  Notice that this is again similar to the distribution of multiplication over addition:
  \[
  \lambda \left(\{\{a:1\}:1,\{b:1,c:1\}:1\} \right) = \{\{a:1,b:1\}:1,\{a:1,c:1\}:1\}
  \]
\end{example}

Not all distributive laws resemble this distributivity of times over plus. Distributive laws need not even be unique. Manes and Mulry present three different distributive laws for the non-empty list monad over itself, none of which are like times over plus:

\begin{example}[Non-Empty List Monad Composed with Itself]
 There are at least three distributive laws for the non-empty list monad over itself:
 \begin{enumerate}[label={[\arabic*]}]
   \item The first distributive law, from \citet{ManesMulry2007}, is given by a syntactic manipulation. It is best illustrated by an example. Given the list of lists $[[a],[b,c,d],[e,f]]$, every comma `$,$' in between two elements is replaced by bracket-comma-bracket `$],[$', and every occurrence of `$],[$' in the list is replaced by a comma `$,$':
       \[
       \lambda ([[a],[b,c,d],[e,f]]) = [[a,b],[c],[d,e],[f]]
       \]
   \item The second distributive law for the non-empty list monad over itself, given in \citet{ManesMulry2008}, is recursively defined as:
   \begin{align*}
     \lambda ([[a_1,\ldots, a_n]]) & = [[a_1],\ldots,[a_n]] \\
     \lambda (\mathcal{L}_1\concat\mathcal{L}_2) & = [\fst(\lambda (\mathcal{L}_1)) \concat \fst(\lambda (\mathcal{L}_2))]
   \end{align*}
    So:
    \begin{align*}
    \lambda ([[a], [b, c], [d,e]]) & = [\fst(\lambda [[a]]) \concat \fst(\lambda [[b,c]]) \concat \fst(\lambda [[d,e]])] \\
    & = [\fst([[a]]) \concat \fst([[b],[c]]) \concat \fst([[d],[e]])] \\
    & = [[a, b, d]]
    \end{align*}
   \item The third distributive law for non-empty list over itself, also from \citet{ManesMulry2008}, is similar to the second, only now instead of consistently taking the first element of the list, the distributive law picks the last element.
 \end{enumerate}
\end{example}

\subsection{Algebraic view}
\label{sec:algebraic-view}
There is an intimate connection between monads and universal algebra~\citep{Lawvere1963, Linton1966, Manes1976}. This connection will be essential to the results in later sections, where we analyze monads and distributive laws from an algebraic perspective. To do this, we exploit the explicit algebraic formulation of distributive laws in terms of~\emph{composite theories}, described in~\citet{PirogStaton2017}. In this section we explain these connections, and establish the notation and vocabulary that we will use in subsequent sections.

\subsubsection{Algebraic Theories}
Firstly, we make precise what we mean by an algebraic theory.
\begin{definition}[Algebraic Theory]
An~\define{algebraic signature} is a set of operation symbols~$\Sigma$, each with an associated natural number referred to as its~\define{arity}. The set of~\define{$\Sigma$-terms} over a set~$X$ of variables is defined inductively as follows:
\begin{itemize}
\item Each element~$x \in X$ is a term.
\item If~$t_1,...,t_n$ are terms, and~$\sigma \in \Sigma$ has arity~$n$, then~$\sigma(t_1,...,t_n)$ is a term.
\end{itemize}
An~\define{algebraic theory}~$\bb{T}$ consists of:
\begin{itemize}
\item An algebraic signature~$\signatureT$.
\item A set of pairs of $\Sigma$-terms over~$X$, $\eqT$, referred to as the \define{equations} or~\define{axioms} of~$\bb{T}$. We will often write a pair~$(s,t) \in \eqT$ as the more readable~$s = t$ when convenient.
\end{itemize}
The following terminology and notational conventions will be used:
\begin{itemize}
\item Operations of arity~$n$, will be referred to as~\define{$n$-ary}. Operations of arity 1,2 and 3 will be referred to as~\define{unary, binary and ternary operations}. 0-ary terms will be referred to as~\define{constants}.
\item Variable contexts: If $Y \subseteq X$ is a set of variables and $t$ is a term, then we write $Y \vdash t$ to mean that the free variables in $t$ are a subset of $Y$. When we wish to make the algebraic theory of interest explicit, we will write $Y \vdash_{\mathbb{T}} t$ meaning~$t$ is a term of the algebraic theory~$\mathbb{T}$ in variable context~$Y$.
\item Set of free variables: If we need the precise set of free variables in a term $t$, we write $\var(t)$. In addition, $\#\var(t)$ denotes the cardinality of the set $\var(t)$.
\item Equality of terms: When two terms $t_1, t_2$ in $\bb{T}$ can be proved equal using equational logic and axioms of $\bb{T}$ as outlined in \reffig~\ref{fig:equational-logic}, we write $t_1 \theoryeq{T} t_2$. When we wish to be specific about which variables may appear in~$t_1$ and~$t_2$, we will write:
  \begin{equation*}
    X \vdash_{\mathbb{T}} t_1 = t_2 \qquad\text{ iff }\qquad X \vdash t_1 \;\text{ and }\; X \vdash t_2 \;\text{ and }\; t_1 \theoryeq{T} t_2
  \end{equation*}
\item Substitution: For a term in context $Y\vdash_\bb{T} t$ and partial function $f$ mapping variables in $Y$ to $\bb{T}$-terms, we will write $t[f]$ or $t[f(y)/y]$ for the corresponding substitution of terms for variables where $f$ is defined.
\item Extended infix and postfix notation: We will extend our notation in the natural way to include infix notation for binary function symbols with symbolic names, for example $a + b$, and postfix notation when it is standard to do so, for example to indicate a multiplicative inverse $a^{-1}$.
\end{itemize}
\end{definition}

\begin{figure}
\figrule
\begin{tabular}{llc}
Axiom: & \;\; &
\AxiomC{$(s,t) \in \eqT$}
\UnaryInfC{$s \theoryeq{T} t$}
\DisplayProof \\
& \;\; & \\
Reflexivity: & \;\; &
\AxiomC{}
\UnaryInfC{$ t \theoryeq{T} t$}
\DisplayProof \\
& \;\; & \\
Symmetry: & \;\; &
\Axiom$ t \fCenter{\ \theoryeq{T}\ } t'$
\UnaryInf$ t' \fCenter{\ \theoryeq{T}\ } t$
\DisplayProof \\
 & \;\; & \\
Transitivity: & \;\; &
\AxiomC{$t \theoryeq{T} t'$}
\AxiomC{$t' \theoryeq{T} t''$}
\BinaryInfC{$ t' \theoryeq{T} t''$}
\DisplayProof \\
 & \;\; & \\
For any $n$-ary operation symbol $\sigma$: & \;\; &
\AxiomC{$t_1 \theoryeq{T} t'_1, \ldots, t_n \theoryeq{T} t'_n$}
\UnaryInfC{$\sigma(t_1,\ldots,t_n) \theoryeq{T} \sigma(t'_1,\ldots,t'_n)$}
\DisplayProof \\
& \;\; & \\
For any substitution f: & \;\; &
\Axiom$ t \fCenter{\ \theoryeq{T}\ } t'$
\UnaryInf$ t[f] \fCenter{\ \theoryeq{T}\ } t'[f]$
\DisplayProof\\
& \;\; &
\end{tabular}

\caption{Inference rules of equational logic}\label{fig:eqlogic}
\figrule
\label{fig:equational-logic}
\end{figure}

We give a few examples of common algebraic theories that will be of particular interest to us.
\begin{example}[Pointed Sets]
  \label{ex:pointed-sets}
  The algebraic theory of pointed sets has a signature with one constant, and no equations.
\end{example}
\begin{example}[Monoids]
  \label{ex:monoids}
  The algebraic theory of monoids has a signature containing one constant $e$ and one binary operation $*$, satisfying the axioms:
  \begin{align}
    e * a & = a \tag{left unit} \\
    a * e & = a \tag{right unit} \\
    (a * b) * c & = a * (b * c) \tag{associativity}
  \end{align}
  The theory of \emph{commutative monoids} extends this theory with one further equation:
  \begin{equation}
    a * b = b * a \tag{commutativity}
  \end{equation}
  The theory of \emph{join semilattices} extends the theory of commutative monoids with the additional axiom:
  \begin{equation}
    a * a = a \tag{idempotence}
  \end{equation}
  These theories are part of a hierarchy of theories called the Boom hierarchy, which will be investigated in detail in \refsec~\ref{section:boom}.
\end{example}
\begin{example}[Abelian Groups]
  \label{ex:abelian-groups}
  The algebraic theory of Abelian groups has an additional operation beyond the signature of monoids. In addition to a constant $e$ and binary operation $*$, it has a unary operation $(\cdot)^{-1}$. The axioms are: left and right unit, associativity, commutativity, and:
  \begin{align}
    a^{-1} * a & = e \tag{left inverse} \\
    a * a^{-1} & = e \tag{right inverse}
  \end{align}
\end{example}

\subsubsection{Algebras and the Free / Forgetful Adjunction}
\label{sec:algebras-and-the-free-forgetful-adjunction}

We now describe the notion of algebra for an algebraic theory. The category of such algebras will lead us to an adjunction of free and forgetful functors, which we need to explain the connection between algebraic theories and monads.

\begin{definition}[$\Sigma$-algebra]
  For a signature~$\Sigma$, a~\define{$\Sigma$-algebra} consists of:
  \begin{itemize}
  \item An underlying set~$A$.
  \item For each~$n$-ary operation~$\sigma \in \signatureT$ a function~$\llbracket \sigma \rrbracket : A^n \rightarrow A$.
  \end{itemize}
  A~\define{homomorphism of~$\Sigma$-algebras} of type~$(A, \llbracket - \rrbracket^A) \rightarrow (B, \llbracket - \rrbracket^B)$ is a function~$h : A \rightarrow B$ such that for each~$n$-ary $\sigma \in \Sigma$:
  \begin{equation*}
    h(\llbracket \sigma \rrbracket^A(x_1,...,x_n)) = \llbracket \sigma \rrbracket^B(h(x_1),...,h(x_n))
  \end{equation*}
  For a fixed~$\Sigma$-algebra, every $\Sigma$-term $\{1, \ldots, n\} \vdash t$ induces a function~$\llbracket t \rrbracket: A^n \rightarrow A$, defined inductively as follows:
  \begin{itemize}
  \item For variable $1 \leq i \leq n$:
    \begin{equation*}
      \llbracket i \rrbracket(x_1,...,x_n) = x_i
    \end{equation*}
  \item For~$m$-ary operation~$\sigma \in \Sigma$:
    \begin{equation*}
      \llbracket \sigma(t_1,...,t_m) \rrbracket(x_1,...,x_n) = \llbracket \sigma \rrbracket(\llbracket t_1 \rrbracket(x_1,...,x_n),...,\llbracket t_m \rrbracket(x_1,...,x_n))
    \end{equation*}
  \end{itemize}
  An algebra \define{satisfies equation~$\{1, \ldots, n\} \vdash s = t$} if:
  \begin{equation*}
    \llbracket s \rrbracket = \llbracket t \rrbracket
  \end{equation*}
\end{definition}
\begin{definition}[$(\Sigma,E)$-algebra]
  For an algebraic theory~$(\Sigma,E)$, a~\define{$(\Sigma, E)$-algebra} is a~$\Sigma$-algebra that satisfies all the equations in~$E$. A~\define{$(\Sigma,E)$-algebra homomorphism} is just a~$\Sigma$-algebra homomorphism between~$(\Sigma,E)$-algebras. $(\Sigma,E)$-algebras and their homomorphisms form a category~$\alge{\Sigma}{E}$.
\end{definition}

\newpage
\begin{example}[Monoids on the Natural Numbers]
  The natural numbers with either addition and~$0$, or multiplication and~$1$, are algebras for the theory of monoids.
\end{example}

\begin{remark}\label{rem:morethanonetheory}
It is not uncommon to have two algebraic theories that give rise to isomorphic categories of algebras. For instance, consider the following algebraic theory, having a signature with exactly one $n$-ary operation $\phi_n$ for each $n \in \bb{N}$, and flattening equations saying:
\begin{equation*}
\phi_n(\phi_{m_1}(x^1_1,...,x^1_{m_1}),...,\phi_{m_n}(x^n_1,...,x^n_{m_n})) = \phi_{\sum_i m_i}(x^1_1,...,x^n_{m_n})
\end{equation*}
Then the category of algebras for this theory is isomorphic to the category of algebras for the theory of monoids.
\end{remark}

Given a set~$A$, there is a canonical way to construct a~$(\Sigma,E)$-algebra on it.
\begin{theorem}
  For an algebraic theory~$\bb{T}$, there is a left adjoint~$F^{\bb{T}}$ to the obvious forgetful functor~$U^{\bb{T}} : \alge{\signatureT}{\eqT} \rightarrow \catset$. The functor~$F^{\bb{T}}$ is defined as follows:
  \begin{itemize}
  \item For set~$A$, $F^{\bb{T}}(A)$ is the set of equivalence classes of~$\signatureT$-terms under provable equality in equational logic, as described in \reffig~\ref{fig:equational-logic}.
  \item The action on a function~$h : A \rightarrow B$ is defined inductively on representatives as follows:
    \begin{align*}
      F^{\bb{T}}(h)(a) &= h(a) \quad\text{ for } a \in A\\
      F^{\bb{T}}(h)(\sigma(t_1,...,t_n)) &= \sigma(F^{\bb{T}}(h)(t_1),...,F^{\bb{T}}(h)(t_n)) \quad\text{ for $n$-ary } \sigma \in \signatureT
    \end{align*}
  \end{itemize}
\end{theorem}

\subsubsection{Algebraic Perspective on Monads}

We now have all the ingredients to describe monads algebraically.

\begin{definition}[Free model monad]
  For an algebraic theory~$\bb{T}$, the~\define{free model monad} induced by~$\bb{T}$ is the monad induced by the free/forgetful adjunction~$U^{\bb{T}} \circ F^{\bb{T}}$.
\end{definition}
It is well known that every finitary monad on the category of sets arises as a free model monad for some algebraic theory. In fact, every monad arises from a generalization of algebraic theories, if we allow infinite arities, see~\citet{Linton1966, Manes1976}. All the monads appearing in the present work are finitary, so we remain in the realm of conventional universal algebra. Note that although every finitary monad will arise from a choice of operations and equations~$(\Sigma,E)$, this choice is not unique, see \refrem~\ref{rem:morethanonetheory}.
\begin{definition}
  We will say that \define{monad~$T$ corresponds to theory~$\bb{T}$} or \define{has presentation/is presented by $\bb{T}$} if~$T$ is a free model monad induced by~$\bb{T}$. In general, a monad has more than one presentation. Even so, if a certain presentation is commonly used for a monad, we refer to it as \emph{the} theory corresponding to the monad.
\end{definition}

\newpage
\begin{example}
  \label{ex:monad-alg-connection}
  Many of the monads we have seen so far correspond to familiar algebraic theories. Some monads are even named in acknowledgment of this correspondence.
  \begin{enumerate}[label={[\arabic*]}]
  \item The lift monad of \refex~\ref{ex:exception-monad} corresponds to the theory of pointed sets described in \refex~\ref{ex:pointed-sets}.
  \item The list, multiset and finite powerset monads of \refexs~\ref{ex:list-monad}, \ref{ex:multiset-monad} and~\ref{ex:powerset-monad} correspond respectively to the theories of monoids, commutative monoids and join semilattices described in \refex~\ref{ex:monoids}.
  \item The Abelian group monad of \refex~\ref{ex:multiset-monad} corresponds to the theory of Abelian groups described in \refex~\ref{ex:abelian-groups}.
  \item The ring monad, which was constructed using a distributive law in \refex~\ref{ex:ring-monad}, corresponds to the theory of rings.
  \end{enumerate}
  Some monads have presentations that are less familiar:
  \begin{enumerate}[label={[\arabic*]}]
    \setcounter{enumi}{4}
  \item The algebraic theory corresponding to the exception monad has a signature containing a constant for each exception, and no axioms.
  \item The algebraic theory corresponding to the distribution monad of \refex~\ref{ex:distribution-monad} is the theory of convex, or barycentric, algebras \citep{Jacobs2010,Stone1949}. These can be described as follows. For each $p \in (0,1)$, the signature contains a binary operation $+^p$, and these satisfy the following axioms:
    \begin{align*}
      x +^p x & = x\\
      x +^p y &= y +^{1-p} x \\
      x +^p (y +^r z) &= (x +^{\frac{p}{p + (1-p)r}} y) +^{p +(1-p)r} z
    \end{align*}
    Note that convex algebras are often described using binary operations for~$p$ in the closed interval~$[0,1]$, the less redundant theory above is equivalent.
  \item The algebraic theory corresponding to the reader monad has a signature containing just an $n$-ary operation, where $n$ is the cardinality of the set of states $R$ used in the reader monad.
  For the case~$n = 2$ we introduce a binary operation~$*$, with the intuitive reading of~$x * y$ being ``if the state is~$1$ do~$x$ else do~$y$'', the extension to larger state spaces has an analogous formulation describing how to proceed conditional on the state that is read. The axioms in the binary case are:
  \begin{align*}
   x * x & = x \\
  (w * x) * (y * z) & = w * z
  \end{align*}
  The first axiom is idempotence and generalises easily to the general case. The second axiom generalises to taking a diagonal. The algebraic formulation of computational monads such as this one is described in~\citet{PlotkinPower2002}.
  \end{enumerate}
\end{example}

\subsubsection{Algebraic Perspective on Distributive Laws}
Given monads $S$ and $T$, a distributive law is used to compose the two monads to form a monad $TS$. Analogously, a composite theory is used to `compose' two algebraic theories $\bb{S}$ and $\bb{T}$ in such a way that each term in the composite theory can be uniquely written as a $\bb{T}$-term over $\bb{S}$-terms.

\begin{definition}[Composite Theories]
  \label{def:composite}
  Let $\bb{S}$ and $\bb{T}$ be two algebraic theories, and let $\bb{U}$ be an algebraic theory that contains both $\bb{S}$ and $\bb{T}$. Then $\bb{U}$ is a \emph{composite} of $\bb{S}$ and $\bb{T}$ if it satisfies the following two conditions:
  \begin{enumerate}
  \item \emph{Separation:} Every term $u$ in $\bb{U}$ is equal to a term $t$ in $\bb{T}$ built from terms $s_i$ in $\bb{S}$: $u \theoryeq{U} t[s_i/x_i]$.
  \item \emph{Essential uniqueness:} The separation of $\bb{U}$-terms into $\bb{T}$-terms of $\bb{S}$-terms is unique modulo $\bb{S}, \bb{T}$. That is, for terms $X \vdash_\bb{T} t$, $X'\vdash_{\bb{T}} t'$ and $Y_i \vdash_{\bb{S}} s_i$, $Y'_{i'} \vdash_\bb{S} s'_{i'}$ such that $t[s_i/x_i] \theoryeq{U} t'[s'_{i'}/x'_{i'}]$, there is a set of variables $Z$ and there are functions $\essuniqfunctionA: X \rightarrow Z$, $\essuniqfunctionB: X' \rightarrow Z$ satisfying:
      \begin{enumerate}
        \item $t[\essuniqfunctionA(x_i)/x_i] \theoryeq{T} t'[\essuniqfunctionB(x'_{i'})/x'_{i'}]$.
        \item For all $i,j$: $\essuniqfunctionA(x_i) = \essuniqfunctionA(x_j) \Leftrightarrow s_i \theoryeq{S} s_j$.
        \item For all $i',j'$: $\essuniqfunctionB(x'_{i'}) = \essuniqfunctionB(x'_{j'}) \Leftrightarrow s'_{i'} \theoryeq{S} s'_{j'}$.
        \item For all $i,j'$: $\essuniqfunctionA(x_i) = \essuniqfunctionB(x'_{j'}) \Leftrightarrow s_i \theoryeq{S} s'_{j'}$.
      \end{enumerate}
  \end{enumerate}
\end{definition}

\begin{theorem}[Pir\'og \& Staton]
  \label{thm:distlaw-vs-compositetheory}
  Let~$S$ and~$T$ be $\catset$-monads with corresponding algebraic theories~$\bb{S}$ and~$\bb{T}$. Every distributive law $ST \Rightarrow TS$ induces a monad that is the free model monad of a composite theory of~$\bb{S}$ and~$\bb{T}$.
\end{theorem}
In particular, if there exists a distributive law $ST \Rightarrow TS$, then there exists an algebraic theory that is a composite of algebraic theories $\bb{S}$ and $\bb{T}$ corresponding to the monads $S$ and $T$. We will frequently use the contrapositive of this statement:
\begin{corollary}\label{cor:noComposite-noDistlaw}
  If there is no composite theory of algebraic theories~$\bb{S}$ and~$\bb{T}$, then there is no distributive law~$ST \Rightarrow TS$ for the corresponding free model monads.
\end{corollary}
The correspondence between distributive laws and composite theories is particularly useful because the axioms of a distributive law are difficult to prove. Depending on the definition of a particular distributive law, applying it to prove the multiplication axioms could require various separate case distinctions and difficult manipulations of one data structure within another. On the algebraic side, the manipulation of terms is restricted to equational logic. This greatly simplifies the proofs.

%% file: plotkin.tex
\section{General Plotkin Theorems}\label{Plotkin}
In this section we develop an algebraic generalization of the counterexample attributed to Gordon Plotkin in~\citet{VaraccaWinskel2006}. This counterexample showed that there is no distributive law of type:
\begin{equation*}
  D \circ P \Rightarrow P \circ D
\end{equation*}
where $D$ is the distribution monad of \refex~\ref{ex:distribution-monad} and~$P$ is the finite powerset monad of \refex~\ref{ex:powerset-monad}. In fact the counterexample is actually shown for what is known as the free real cone or finite valuation monad, as this requires slightly weaker assumptions. We state it here for the distribution monad simply to avoid the distraction of introducing yet another monad. This is essentially a cosmetic decision, and our results are equally applicable to the original counterexample for the free real cone monad.

Our aim in this section is to extract general methods for showing no-go results for constructing distributive laws, derived from the essentials of why this counterexample works.
In order to do this, we isolate sufficient conditions on algebraic theories inducing two monads, such that there can be no distributive law between them. Earlier generalizations of this counterexample have appeared in~\citet{KlinSalamanca2018, DahlqvistNeves2018}. All the existing approaches involve direct calculations involving the distributive law axioms, leading to somewhat opaque conditions. They also remain limited to the case where one of the two monads is the powerset monad, restricting their scope of application.

\subsection{The Original Counterexample}
\label{sec:plotkin-classic}
This section presents a very slight rephrasing of Plotkin's original counterexample, as presented in~\citet{VaraccaWinskel2006}, adjusted to the case of the powerset and distribution monads used in this paper. We will augment the proof with commentary aimed at providing motivation for the main proof ideas. Hopefully this will provide a useful stepping stone towards understanding the general theorems of \refsec~\ref{sec:general-theorems}.

\begin{counter}[Probability does not distribute over non-determinism]
  Assume, for contradiction, that there is a distributive law of type~$\lambda : DP \Rightarrow PD$. Fix the set~$X = \{ a, b, c, d \}$, and consider the element~$\plotkinspecial \in DP(X)$ defined by:
  \begin{equation*}
    \plotkinspecial = \{ a, b \} +^{\frac{1}{2}} \{c, d \}
  \end{equation*}
  We introduce three functions on our underlying set~$X$:
  \begin{trivlist}\item
    \begin{minipage}{0.3\textwidth}
      \begin{align*}
        f_1 : X &\rightarrow X\\
        a &\mapsto a\\
        b &\mapsto b\\
        c &\mapsto a\\
        d &\mapsto b
      \end{align*}
    \end{minipage}
    \begin{minipage}{0.3\textwidth}
      \begin{align*}
        f_2 : X &\rightarrow X\\
        a &\mapsto a\\
        b &\mapsto b\\
        c &\mapsto b\\
        d &\mapsto a
      \end{align*}
    \end{minipage}
    \begin{minipage}{0.3\textwidth}
      \begin{align*}
        f_3 : X &\rightarrow X\\
        a &\mapsto a\\
        b &\mapsto a\\
        c &\mapsto c\\
        d &\mapsto c
      \end{align*}
    \end{minipage}
  \end{trivlist}
  The plan of the proof is to analyze how~$\plotkinspecial$ travels around the naturality square for~$\lambda$, for each of our three functions. The element~$\plotkinspecial$ and the three functions have been carefully chosen so that the distributive law unit axioms can be applied during the proof.
  \begin{equation}
    \label{eq:plotkin-key}
    \begin{gathered}
      \begin{tikzpicture}[scale=0.5, node distance=2cm, ->]
        \node (tl) {$DP(X)$};
        \node[below of=tl] (bl) {$DP(X)$};
        \node[right of=tl] (tr) {$PD(X)$};
        \node[right of=bl] (br) {$PD(X)$};
        \draw (tl) to node[above]{$\lambda_X$} (tr);
        \draw (bl) to node[below]{$\lambda_X$} (br);
        \draw (tl) to node[left]{$DP(f_i)$} (bl);
        \draw (tr) to node[right]{$PD(f_i)$} (br);
      \end{tikzpicture}
    \end{gathered}
  \end{equation}
  We proceed as follows:
  \begin{itemize}
  \item Trace~$\plotkinspecial$ around the naturality square~\eqref{eq:plotkin-key} for both~$f_1$ and~$f_2$. We note that as~$\{-,-\}$ is commutative, and~$+^{\frac{1}{2}}$ is idempotent:
    \begin{equation*}
      DP(f_1)(\plotkinspecial) = \eta^D_{PX}\{ a, b \} = DP(f_2)(\plotkinspecial)
    \end{equation*}
    Idempotence and commutativity will be important ideas for any general theorem. For both~$f_1$ and~$f_2$ we can apply one of the distributive law unit axioms to conclude that:
    \begin{equation*}
      \lambda_X \circ DP(f_1)(\plotkinspecial) = \{ \eta^D_X(a), \eta^D_X(b) \} = \lambda_X \circ DP(f_2)(\plotkinspecial)
    \end{equation*}
    Now a careful consideration of the actions~$PD(f_1)$ and~$PD(f_2)$ allows us to deduce that~$\lambda_X(\plotkinspecial)$ must be a subset of:
    \begin{equation*}
      \{ \eta^D_X(a), \eta^D_X(b), \eta^D_X(c), \eta^D_X(d) \}
    \end{equation*}
    This step is one of more awkward components to generalize. In principle it involves inverse images of equivalence classes of terms in one algebraic theory, with variables labelled by equivalence classes of terms in a second algebraic theory, and motivates our move to an explicitly algebraic formulation. We can see this whole step as establishing an upper bound on the set of variables appearing in~$\lambda_X(\plotkinspecial)$.
  \item We then trace~$\plotkinspecial$ around the naturality square~\eqref{eq:plotkin-key} for~$f_3$. In this case, we exploit the idempotence of the operation~$\{-, -\}$ to conclude:
    \begin{equation}
      \label{eq:plotkin-intermediate}
      DP(f_3)(\plotkinspecial) = \{a\} +^{\frac{1}{2}} \{ c \}
    \end{equation}
    This indicates idempotence is actually an important aspect of both monads for this argument to work. Equation~\eqref{eq:plotkin-intermediate} allows us to apply the second unit axiom for~$\lambda$ to conclude:
    \begin{equation*}
      \lambda_X \circ DP(f_3)(\plotkinspecial) = \{ a +^{\frac{1}{2}}  c \}
    \end{equation*}
    By considering the action of~$PD(f_3)$ as before, we conclude that~$\lambda_X(\plotkinspecial)$ must contain an element mapped onto~$a +^{\frac{1}{2}} c$ by~$PD(f_3)$, placing a lower bound on the set of variables that appear in~$\lambda_X(\plotkinspecial)$.
  \item The lower and upper bounds established in the previous two steps contradict each other, and so no distributive law of type~$DP \Rightarrow PD$ can exist.
  \end{itemize}
\end{counter}
In summary, the argument requires two components:
\begin{itemize}
\item Some operations satisfying certain algebraic equational properties such as idempotence and commutativity.
\item Some slightly more mysterious properties of our monads, making the ``inverse image'' parts of the argument work correctly.
\end{itemize}
The axioms in \refsec~\ref{sec:general-theorems} distill these imprecise intuitions into concrete axioms that can be used to yield general theorems.
\subsection{The General Theorems}
\label{sec:general-theorems}
We now move to our generalizations of the Plotkin counterexample outlined in \refsec~\ref{sec:plotkin-classic}, adopting many of the notational conventions of~\citet{VaraccaWinskel2006, PirogStaton2017} whenever possible to ease comparison with those papers. The composite theories as described in \refdef~\ref{def:composite} will be used heavily, as will the algebraic characterization of distributive laws of \refthm~\ref{thm:distlaw-vs-compositetheory}.

We introduce terminology for some special sets of terms in an algebraic theory. The theorems are stated in terms of these special sets, which in some cases can restrict the scope for which certain ``global'' conditions need to apply, broadening the range of applicability.
\begin{definition}[Universal Terms]
  For an algebraic theory, we say that a set of terms~$T$ is:
  \begin{itemize}
  \item \define{Universal} if every term is provably equal to a term in $T$.
  \item \define{Stable} if $T$ is closed under substitution of variables for variables.
  \end{itemize}
\end{definition}
\begin{example}
  Some examples of universal and stable sets:
  \begin{enumerate}[label={[\arabic*]}]
  \item For any theory, the set of all terms is a stable universal set.
  \item For the theory of real vector spaces, every term is equal to a term in which scaling by the zero element does not appear. Terms that do not contain the scale by zero operation are clearly also stable under variable renaming. Therefore the terms not containing the scale by zero operation are a stable universal set.
  \item In the theory of groups, every term is equal to a term in which no subterm and its inverse are ``adjacent''. This set is therefore universal. It is not stable, as variable renaming may introduce a subterm adjacent to its inverse.
  \end{enumerate}
\end{example}
\begin{remark}
  On first reading, it is probably easiest to consider \refthms~\ref{thm:plotkin-1} and~\ref{thm:plotkin-2} with the universal algebraic sets taken to be the set of all terms in a theory, as this is by far the most common case.
\end{remark}

Throughout this section, the variable labels for any algebraic theory will range over the natural numbers.
We will also write~$n$ for the set~$\{1,...,n\}$, so, for example:
\begin{equation*}
  2 \vdash t
\end{equation*}
means~$t$ is a term containing (at most) two variables.

We proceed in two steps. Theorem~\ref{thm:plotkin-1} is an algebraic generalization of Plotkin's counterexample, capturing the algebraic properties required of both theories in order for a proof of this type to work. In \refthm~\ref{thm:plotkin-2} we generalize further, removing the restriction to binary terms that was sufficient for the original application. This generalization complicates the proof slightly, and so to clarify the methods involved we present two separate \refthms.

\begin{theorem}
  \label{thm:plotkin-1}
  Let~$\mathbb{P}$ and~$\mathbb{V}$ be two algebraic theories, $T_{\mathbb{P}}$ a stable universal set of ~$\mathbb{P}$-terms, and~$T_{\mathbb{V}}$ a stable universal set of $\mathbb{V}$-terms. If there are terms:
  \begin{equation*}
    2 \vdash_{\mathbb{P}} p \qquad\text{ and }\qquad 2 \vdash_{\mathbb{V}} v
  \end{equation*}
  such that:
  \begin{enumerate}[label=(P\arabic*)]
  \item \label{ax:pcomm} $p$ is commutative:
    \begin{equation*}
      2 \vdash p(1,2) = p(2,1)
    \end{equation*}
  \item \label{ax:pidem} $p$ is idempotent:
    \begin{equation*}
      1 \vdash p(1,1) = 1
    \end{equation*}
  \item \label{ax:pprolif} For all~$p' \in T_{\mathbb{P}}$:
    \begin{equation*}
      n \vdash p(1,2) = p' \quad\Rightarrow\quad 2 \vdash p'
    \end{equation*}
  \end{enumerate}
  \begin{enumerate}[label=(V\arabic*)]
  \item \label{ax:videm} $v$ is idempotent:
    \begin{equation*}
      1 \vdash v(1,1) = 1
    \end{equation*}
  \item \label{ax:vvar} For all~$v' \in T_{\mathbb{V}}$, and~$x$ a variable:
    \begin{equation*}
      \Gamma \vdash_{\mathbb{V}} x = v' \quad\Rightarrow\quad \{x\} \vdash v'
    \end{equation*}
  \item \label{ax:vbinary} For all~$v' \in T_{\mathbb{V}}$:
    \begin{equation*}
      \Gamma \vdash v(1,2) = v' \quad\Rightarrow\quad \neg (\{ 1 \} \vdash v'\vee \{ 2 \} \vdash v')
    \end{equation*}
  \end{enumerate}
  Then there is no composite theory of~$\mathbb{P}$ and~$\mathbb{V}$.
\end{theorem}

\begin{remark}
Properties \ref{ax:pprolif}, \ref{ax:vvar}, and \ref{ax:vbinary} are constraints on the variables appearing in certain terms, which is needed for the ``inverse image'' part of the Plotkin argument to work. Property \ref{ax:pprolif} states that any term equal to the special binary term $p$ can have at most two free variables. Property \ref{ax:vvar} states that any term equal to a variable can only contain that variable, and property \ref{ax:vbinary} states that any term equal to the special binary term $v$ must have at least two free variables. Notice that the upper/lower bound principle from the argument is reflected in these conditions.
\end{remark}

\begin{proof}
  Assume for contradiction that a composite theory~$\mathbb{U}$ of~$\mathbb{P}$ and~$\mathbb{V}$ exists. Then as~$\mathbb{U}$ is composite, there exist~$X \vdash p'$ and~$4 \vdash v'_i$ such that:
  \begin{equation*}
    4 \vdash_{\mathbb{U}} v(p(1,2),p(3,4)) = p'[v'_i/x_i]
  \end{equation*}
  Without loss of generality, we may assume~$p' \in T_{\mathbb{P}}$ and~$v'_i \in T_{\mathbb{V}}$ by universality.  Define partial function~$f_1$ as follows:
  \begin{align*}
    f_1(1) &= f_1(3) = 1\\
    f_1(2) &= f_1(4) = 2
  \end{align*}
  then using assumption~\ref{ax:videm}:
  \begin{equation*}
    4 \vdash p(1,2) = p'[v'_i[f_1]/x_i]
  \end{equation*}
  As~$T_\mathbb{P}$ is stable, any variable renaming of~$p'$ is also in~$T_{\mathbb{P}}$. Therefore by essential uniqueness and assumption~\ref{ax:pprolif}:
  \begin{equation*}
    4 \vdash 1 = v'_i[f_1] \quad\vee\quad 4 \vdash 2 = v'_i[f_1]
  \end{equation*}
  Then using assumption~\ref{ax:vvar}, for all~$v'_i$:
  \begin{equation}
    \label{eq:condition1}
    \{1,3\} \vdash v'_i \quad\vee\quad \{2,4\} \vdash v'_i
  \end{equation}
  Define a second partial function~$f_2$ as follows:
  \begin{align*}
    f_2(1) &= f_2(4) = 1\\
    f_2(2) &= f_2(3) = 2
  \end{align*}
  By assumptions~\ref{ax:pcomm} and~\ref{ax:videm}:
  \begin{equation*}
    4 \vdash p(1,2) = p'[v'_i[f_2]/x_i]
  \end{equation*}
  Again, as~$T_\mathbb{P}$ is stable, any variable renaming of~$p'$ is also in~$T_{\mathbb{P}}$. Therefore by essential uniqueness and assumption~\ref{ax:pprolif}:
  \begin{equation}
    4 \vdash 1 = v'_i[f_2] \quad\vee\quad 4 \vdash 2 = v'_i[f_2]
  \end{equation}
  and so, using assumption~\ref{ax:vvar}, for all~$v'_i$:
  \begin{equation}
    \label{eq:condition2}
    \{1,4\} \vdash v'_i \qquad\vee\qquad \{2,3\} \vdash v'_i
  \end{equation}
  Combining conditions~\eqref{eq:condition1} and~\eqref{eq:condition2}, for all~$v'_i$:
  \begin{equation}
    \label{eq:condition3}
    \bigvee_{n \in 4} \{n\} \vdash v'_i
  \end{equation}
  Finally, we define a third partial function~$f_3$ as follows:
  \begin{align*}
    f_3(1) &= f_3(2) = 1\\
    f_3(3) &= f_3(4) = 2
  \end{align*}
  Using~\ref{ax:pidem}:
  \begin{equation*}
    4 \vdash v(1,3) = p'[v'_i[f_3] / x_i]
  \end{equation*}
  Using essential uniqueness (and consistency of~$\mathbb{P}$), there exists~$i$ such that:
  \begin{equation*}
    4 \vdash v(1,3) = v'_i[f_3]
  \end{equation*}
  As~$T_{\mathbb{V}}$ is stable, each~$v'[f_3]$ is a element of~$T_{\mathbb{V}}$, and so by~\ref{ax:vbinary} this contradicts \refeq~\eqref{eq:condition3}. Therefore the assumed composite theory cannot exist.
\end{proof}
Really the point of \refthm~\ref{thm:plotkin-1} is to describe algebraic conditions under which distributive laws will not exist. The following corollary makes this explicit:
\begin{corollary}
  If monads~$P$ and~$V$ have presentations~$\mathbb{P}$ and~$\mathbb{V}$ such that the conditions of \refthm~\ref{thm:plotkin-1} can be satisfied, then there is no distributive law:
  \begin{equation*}
    VP \Rightarrow PV
  \end{equation*}
\end{corollary}
\begin{proof}
  Immediate from~\cite[Theorem 5]{PirogStaton2017}.
\end{proof}
\begin{example}[Powerset and Distribution Monads]
Consider the term $1 \vee 2$ in the theory of join semilattices, corresponding to the powerset monad, and the term $1 +^{\frac{1}{2}} 2$ in the theory of convex algebras, corresponding to the finite probability distribution monad. Since both of these terms are binary, commutative, and idempotent, \refthm~\ref{thm:plotkin-1} captures the following previously known results:
\begin{itemize}
  \item There is no distributive law $DP \Rightarrow PD$ \citep{VaraccaWinskel2006}.
  \item There is no distributive law $PP \Rightarrow PP$ \citep{KlinSalamanca2018}.
\end{itemize}
In addition, \refthm~\ref{thm:plotkin-1} yields the following new results, completing the picture:
\begin{itemize}
  \item There is no distributive law $PD \Rightarrow DP$.
  \item There is no distributive law $DD \Rightarrow DD$.
\end{itemize}
\end{example}
\begin{example}[Powerset and Distribution Monads again]
  We can also consider the distribution monad to be presented by binary operations~$+^p$ with~$p$ in the~\emph{closed} interval~$[0,1]$, and in fact this is the more common formulation. In this case, \refthm~\ref{thm:plotkin-1} can still be directly applied, without having to move to the more parsimonious presentation. We simply note that the terms not involving the operations~$+^1$ and~$+^0$ form a stable universal set satisfying the required axioms. The results discussed in the previous example can then be recovered using the conventional presentation of the distribution monad.
\end{example}
\begin{counter}[Reader Monad]
It is well known that the Reader monad distributes over itself. Looking at the presentation of the reader monad given in \refex~\ref{ex:monad-alg-connection}, we see that although it has idempotent terms, there is no commutative term and hence \refthm~\ref{thm:plotkin-1} does not apply.
\end{counter}
A natural question to ask with regard to \refthm~\ref{thm:plotkin-1} is whether the choice of~\emph{binary} terms for both~$p$ and~$v$ is necessary. We thank Prakash Panangaden for posing this question during an informal presentation of an earlier version of this work~\citep{Panangaden2018}. The answer is that we can generalize to terms
with any arities strictly greater than one. Before presenting the appropriate generalization of \refthm~\ref{thm:plotkin-1}, we require a small, but important, technical lemma.

In \refsec~\ref{sec:plotkin-classic} we described our intuition for the first part of Plotkin's proof as an attempt to establish an upper bound on the variables appearing in certain terms. In \refthm~\ref{thm:plotkin-2} we apply this tactic again. Well-chosen substitutions will give us a family of sets giving upper bounds on the set of variables that may appear in a certain special term. The crucial \reflem~\ref{lem:filter} exploits the structure of these sets to combine this family of upper bounds into a single, tighter bound.

\begin{lemma}
  \label{lem:filter}
  Let~$n,m$ be strictly positive natural numbers, and~$\sigma$ a fixed point free permutation of~$\{1,...,m\}$. For distinct variables~$a^j_i$, $1 \leq i \leq m$, $1 \leq j \leq n$, the sets:
  \begin{align*}
    &\{ a_{i_1}^1,a_{i_1}^2, a_{i_1}^3, \ldots, a_{i_1}^n \}\\
    &\{ a_{i_2}^1,a_{\sigma (i_2)}^2, a_{i_2}^3, \ldots, a_{i_2}^n \}\\
    &\qquad \quad \;\; \vdots\\
    &\{ a_{i_n}^1,a_{i_n}^2, a_{i_n}^3, \ldots, a_{\sigma (i_n)}^n \}
  \end{align*}
  have at most one common element. Here, each $i_k$ is an element of $\{i \mid 1 \leq i \leq m\}$, not necessarily unique.
\end{lemma}

\newpage
\begin{proof}
  We proceed by induction on~$n$. The base case $n = 1$ is trivially true. For $n = n' + 1$, we consider the first two rows of our table of sets. There are two cases.
  \begin{enumerate}
    \item If~$i_1 = \sigma(i_2)$, then the first two rows can only agree at their second element by the assumption that~$\sigma$ is fixed point free, and the claim follows.
    \item If~$i_1 \neq \sigma(i_2)$ then the first two rows disagree in the second column. Therefore the elements common to all the sets cannot appear in the second column.
      We then remove both row and column 2, and invoke the induction hypothesis for~$n = n'$.
  \end{enumerate}
\end{proof}

We now present our widest generalization of Plotkin's original counterexample.
\begin{theorem}
  \label{thm:plotkin-2}
  Let~$\mathbb{P}$ and~$\mathbb{V}$ be two algebraic theories, $T_{\mathbb{P}}$ a stable universal set of~$\mathbb{P}$-terms, and~$T_{\mathbb{V}}$ a stable universal set of $\mathbb{V}$-terms. If there are terms:
  \begin{equation*}
    m \vdash_{\mathbb{P}} p \qquad\text{ and }\qquad n \vdash_{\mathbb{V}} v
  \end{equation*}
  and a fixed point free permutation~$\sigma: m \rightarrow m$ such that:
  \begin{enumerate}[label=(P\arabic*)]
  \item \label{ax:pcomm-2} $p$ is stable under the permutation~$\sigma$:
    \begin{equation*}
      m \vdash p = p[\sigma]
    \end{equation*}
  \item \label{ax:pidem-2} $p$ is idempotent:
    \begin{equation*}
      1 \vdash p[1/i] = 1
    \end{equation*}
  \item \label{ax:pprolif-2} For all~$p' \in T_{\mathbb{P}}$:
    \begin{equation*}
      \Gamma \vdash p = p' \quad\Rightarrow\quad m \vdash p'
    \end{equation*}
  \end{enumerate}
  \begin{enumerate}[label=(V\arabic*)]
  \item \label{ax:videm-2} $v$ is idempotent:
    \begin{equation*}
      1 \vdash v[1/i] = 1
    \end{equation*}
  \item \label{ax:vvar-2} For all~$v' \in T_{\mathbb{V}}$, and~$x$ a variable:
    \begin{equation*}
      \Gamma \vdash_{\mathbb{V}} x = v' \quad\Rightarrow\quad x \vdash v'
    \end{equation*}
  \item \label{ax:vbinary-2} For all~$v' \in T_{\mathbb{V}}$:
    \begin{equation*}
      \Gamma \vdash v = v' \quad\Rightarrow\quad \neg \left(\bigvee_{i \in \Gamma} \{ i \} \vdash v' \right)
    \end{equation*}
  \end{enumerate}
  Then there is no composite theory of~$\mathbb{P}$ and~$\mathbb{V}$.
\end{theorem}

\begin{remark}
Again, properties \ref{ax:pprolif-2}, \ref{ax:vvar-2}, and \ref{ax:vbinary-2} are constraints on the variables appearing in certain terms. Property \ref{ax:pprolif-2} states that any term equal to the special term $p$ can have at most $m$ free variables, \refpro~\ref{ax:vvar-2} states that any term equal to a variable can only contain that variable, and property \ref{ax:vbinary-2} states that any term equal to the special term $v$ must have at least two free variables.
\end{remark}

\newpage
\begin{proof}
  Assume for contradiction that a composite theory~$\mathbb{U}$ of~$\mathbb{P}$ and~$\mathbb{V}$ exists. Let~$a_i^j$, $1 \leq i \leq m$, $1 \leq j \leq n$ denote distinct natural numbers.
  Then as~$\mathbb{U}$ is composite, there exist~$X \vdash p'$ and~$m * n \vdash v'_i$ such that:
  \begin{equation*}
    m * n \vdash_{\mathbb{U}} v(p(a_1^1, \ldots, a_m^1), \ldots, p(a_1^n, \ldots, a_m^n)) = p'[v'_l/x_l]
  \end{equation*}
  Without loss of generality, we may assume~$p' \in T_{\mathbb{P}}$ and~$v'_i \in T_{\mathbb{V}}$ by universality.

  Define substitution~$f_1$ as follows:
  \begin{equation*}
    f_1(a_i^j) = a_i^1
  \end{equation*}
  We then have:
  \begin{equation*}
    m * n \vdash_{\mathbb{U}} v(p(a_1^1, \ldots, a_m^1), \ldots, p(a_1^1, \ldots, a_m^1)) = p'[v'_l[f_1]/x_l]
  \end{equation*}
  By assumption~\ref{ax:videm-2}
  \begin{equation*}
    m * n \vdash_{\mathbb{U}} p(a_1^1, \ldots, a_m^1) = p'[v'_l[f_1]/x_l]
  \end{equation*}
  As~$T_\mathbb{P}$ is stable, any variable renaming of~$p'$ is also in~$T_{\mathbb{P}}$. Therefore by essential uniqueness and~\ref{ax:pprolif-2}:
  \begin{equation*}
    \forall l \; \exists i:\; m * n \vdash a_i^1 = v'_l[f_1]
  \end{equation*}
  By assumption~\ref{ax:vvar-2}:
  \begin{equation*}
  \forall l \; \exists i:\; \{a_i^1\} \vdash v'_l[f_1]
  \end{equation*}
  And so:
  \begin{equation}
    \label{eq:filter-1}
    \forall l \; \exists i: \; \{ a_i^1, \ldots, a_i^n \} \vdash v'_l
  \end{equation}
  Now we define a family of substitutions for~$2 \leq k \leq n$ as follows:
  \begin{equation*}
    f_k(a_i^j) =
    \begin{cases}
      a_{\sigma(i)}^k \text{ if } j = k\\
      a_i^k \text{ otherwise }
    \end{cases}
  \end{equation*}
  If we follow a similar argument as before, also exploiting assumption~\ref{ax:pcomm-2}, we conclude that:
  \begin{equation*}
  \forall l,k \; \exists i_k: \{a^k_{i_k}\} \vdash v'_l[f_k]
  \end{equation*}
  And so:
  \begin{equation}
    \label{eq:filter-2}
    \forall l,k \; \exists i_k: \; \{ a_{\sigma^{-1} (i_k)}^j \mid j = k \} \cup \{ a_{i_k}^j \mid j \neq k \} \vdash v'_l
  \end{equation}
  Then we note that by \reflem~\ref{lem:filter}, conditions~\eqref{eq:filter-1} and~\eqref{eq:filter-2} that:
  \begin{equation}
    \label{eq:filtered}
    \forall l \; \exists i,j: \; \{ a_i^j \} \vdash v'_l
  \end{equation}
  Define another substitution:
  \begin{equation*}
    f_{n+1}(a_i^j) = a_1^j
  \end{equation*}
  Applying this substitution:
  \begin{equation*}
    m * n \vdash v(p(a_1^1, \ldots, a_1^1), \ldots, p(a_1^n, \ldots, a_1^n)) = p'[v'_l[f_{n+1}] / x_l ]
  \end{equation*}
  Using assumption~\ref{ax:pidem-2}:
  \begin{equation*}
    m * n \vdash v(a_1^1, \ldots, a_1^n) = p'[v'_l[f_{n+1}] / x_l ]
  \end{equation*}
  By essential uniqueness and consistency:
  \begin{equation*}
    \exists l: \; m * n \vdash v(a_1^1, \ldots, a_1^n) = v'_l[f_{n+1}]
  \end{equation*}
  As~$T_{\mathbb{V}}$ is stable, each~$v'[f_{n+1}]$ is a element of~$T_{\mathbb{V}}$, and so by assumption~\ref{ax:vbinary-2}, $v'_l$ must contain at least two variables, but this contradicts conclusion~\eqref{eq:filtered}, and so the assumed composite theory cannot exist.
\end{proof}
It is clear that the simpler \refthm~\ref{thm:plotkin-1} is a special case of \refthm~\ref{thm:plotkin-2}. Beyond simply providing greater generality, the main point of \refthm~\ref{thm:plotkin-2} is that it clearly demonstrates that there is nothing special about binary terms. This further clarifies our understanding of what abstract properties make the original Plotkin counterexample work. By moving to such a high level of abstraction, it is also easier to see that our second method, described in \refsec~\ref{section:beyondPlotkin} is not simply a further generalization of Plotkin's counterexample, as it makes fundamentally different assumptions upon the underlying algebraic theories.

%% file: zwart.tex
\section{No-Go Theorems beyond Plotkin}\label{section:beyondPlotkin}

So far, all negative results about monad compositions involve monads with an idempotent term in their corresponding algebraic theories. But the absence of an idempotent term does not guarantee that a distributive law exists. We will derive several new no-go theorems by turning our attention towards theories that have constants satisfying unit equations. In addition, we will see that some of the existing distributive laws are unique.

The motivation for this section is the question of whether the list monad distributes over itself, to which we give a negative answer. We will use this monad as a leading example throughout the section, although the material we present is far more general.

We will again use the algebraic method, starting from the assumption that there is a composite theory of two algebraic theories. The properties we require from our algebraic theories are comparable to the ones we needed in the generalisation of Plotkin's proof. Some properties will limit the variables that may appear in certain terms, while others require specific equations to hold in our theory. Of special interest is the \define{abides} equation:
\begin{equation}\label{eq:abides}
  (a * b) * (c * d) = (a * c) * (b * d)
\end{equation}
We will be looking at theories in which this equation does \emph{not} hold.

More specifically, if $\bb{S}$ and $\bb{T}$ are the algebraic theories of interest, then for $\bb{S}$ we will need various combinations of the following properties:

\begin{enumerate}[label=(S\arabic*)]
  \item \label{ax:svar0} For any two terms $s_1, s_2$: 
      \begin{equation*}
      \emptyset \vdash s_1 \; \wedge \; \Gamma \vdash s_1 = s_2 \quad\Rightarrow\quad \emptyset \vdash s_2
      \end{equation*}
  \item \label{ax:svar1} For any term $s$ and variable $x$: 
  \[
  \Gamma \vdash s = x \quad\Rightarrow\quad \{x\} \vdash s
  \]
  \item \label{ax:sunits} For every $n$-ary operation $\anyopS_n$ ($n \geq 1$) in the signature $\signatureS$, there is a constant $e_{\anyopS}$, which acts as a unit for $\anyopS_n$. If $f$ is the substitution $x_i \mapsto e_\anyopS$ for all but one $i \neq n$, then:
      \[
      \{x\} \vdash \anyopS_n[f] = x
      \]
  \item $\bb{S}$ has a binary term $\specialopS$ such that:
   \begin{enumerate}
     \item \label{ax:sbinary} $e_\specialopS$ is a unit for $\specialopS$:
      \[
      \{x\} \vdash \specialopS(x,e_\specialopS) = x = \specialopS(e_\specialopS,x)
      \]
     \item \label{ax:sidem} $\specialopS$ is idempotent:
     \[
     \{x\} \vdash \specialopS(x,x) = x
     \]
   \end{enumerate}
\end{enumerate}

The first two properties are constraints on the variables appearing in terms: terms equal to a term with no variables cannot have any variables themselves, while terms equal to a single variable must only contain that variable. The other properties require the units or idempotence equations to hold. Notice that property \ref{ax:sunits} implies that for any unary operation $\anyopS_1$ in signature $\signatureS$, we must have $\anyopS_1(s) \theoryeq{S} s$.

For $\bb{T}$, we will assume combinations of these properties:
\begin{enumerate}[label=(T\arabic*)]
  \item \label{ax:tvar0} For any two terms $t_1, t_2$: 
      \[
      \emptyset \vdash t_1 \; \wedge \; \Gamma \vdash t_1 = t_2 \quad\Rightarrow\quad \emptyset \vdash t_2
      \]
  \item \label{ax:tvar1} For any term $t$ and variable $y$: 
  \[
  \Gamma \vdash t = y \quad\Rightarrow\quad \{y\} \vdash t
  \]
  \item \label{ax:tconst} $\bb{T}$ has a constant $e_\specialopT$.
  \item $\bb{T}$ has a binary term $\specialopT$ such that:
    \begin{enumerate}
      \item\label{ax:tunit} $e_\specialopT$ is a unit for $\specialopT$:
      \[
      \{y\} \vdash \specialopT(y,e_\specialopT) = y = \specialopT(e_\specialopT,y)
      \]
      \item\label{ax:tspecialproperty} The abides equation does not hold in $\bb{T}$:
   \[
   \Gamma \vdash \specialopT(\specialopT(y_1,y_2),\specialopT(y_3,y_4)) = \specialopT(\specialopT(y_1,y_3),\specialopT(y_2,y_4)) \quad\Rightarrow\quad \#\Gamma \leq 3.
   \]
   \end{enumerate}
\end{enumerate}

Again the first two properties are constraints on the variables in terms, while the others concern equations.

\begin{example}\label{ex:STproperties}
  \begin{enumerate}[label={[\arabic*]}]
    \item \emph{List monad:} The algebraic theory of monoids, which corresponds to the list monad, has a binary operation $*$ and a constant $e$, satisfying the unit equations, see \refex~\ref{ex:monoids}. It is easy to check that this theory satisfies \ref{ax:svar0} and \ref{ax:svar1}. Since $e$ is indeed the unit for $*$, it also satisfies \ref{ax:sunits}, and the binary term for \ref{ax:sbinary} is easily found: $x * y$. The theory of monoids does not satisfy \ref{ax:sidem}, as none of the binary terms are idempotent.

        The only way for a monoid to satisfy $(x * y) * (z * w) = (x * z) * (y * w)$ is to have $y = z$, and so the theory satisfies all of \ref{ax:tvar0}, \ref{ax:tvar1}, \ref{ax:tconst}, \ref{ax:tunit}, and \ref{ax:tspecialproperty}.
    \item \emph{Powerset monad:} The algebraic theory of join semilattices, which corresponds to the powerset monad, describes commutative and idempotent monoids. Therefore, it satisfies \ref{ax:sidem} in addition to \ref{ax:svar0}, \ref{ax:svar1}, \ref{ax:sunits}, and \ref{ax:sbinary}. However, because of the additional commutativity, join semilattices satisfy abides:
        \begin{align*}
         (x * y) * (z * w) & = x * ((y * z) * w) \\
           & = x * ((z * y) * w) \\
           & = (x * z) * (y * w)
        \end{align*}
        Therefore, this theory does not have property \ref{ax:tspecialproperty}. Properties \ref{ax:tvar0}, \ref{ax:tvar1}, \ref{ax:tconst}, and \ref{ax:tunit} still hold.
    \item \emph{Exception monad:} The algebraic theory corresponding to the exception monad, from \refex~\ref{ex:monad-alg-connection}, satisfies \ref{ax:svar0}, \ref{ax:svar1}, and \ref{ax:sunits}. But since there are no binary terms, it does not satisfy \ref{ax:sbinary} and \ref{ax:sidem}. Similarly, it satisfies \ref{ax:tvar0}, \ref{ax:tvar1}, and \ref{ax:tconst}, but not \ref{ax:tunit} and \ref{ax:tspecialproperty}.
  \end{enumerate}
\end{example}

Given these properties, we will first consider when an $\bb{S}$-term of $\bb{T}$-terms is equal to the constant promised by \refpro~\ref{ax:tconst}, yielding \refthm~\ref{Prop1}. For the list monad, this means that if a distributive law would exist, it would have to map any list of lists containing the empty list to the empty list:
\[
\lambda [L_1,\ldots,L_n] = [] \text{ if there is an } i \text{ s.t. } L_i = []
\]
If the theory $\bb{T}$ has more than one constant, this immediately leads to a contradiction, and hence gives us the no-go theorem formulated in \refcor~\ref{thm:nogoConstants}.

Then building on \refthm~\ref{Prop1}, we derive conditions in which a distributive law has to behave like the distributivity of times over plus (\refeq~\eqref{timesoverplus}), resulting in \refthm~\ref{theorem}. Again this theorem holds for the list monad, so we see there can be at most one distributive law for the list monad over itself. It follows that if there is a distributive law between two monads satisfying the criteria of this theorem, then it is unique.

Finally, in \refthms~\ref{thm:nogoTreeList} and \ref{thm:nogoTreeIdem} we identify properties that together with \refthm~\ref{theorem} lead to a contradiction, giving us two more no-go theorems. The first property is a lack of the abides equation, which yields \refthm~\ref{thm:nogoTreeList}. From this theorem we conclude that there is no distributive law for the list monad over itself. Theorem \ref{thm:nogoTreeIdem} requires idempotence, and proves that there is no distributive law $PM \Rightarrow MP$ for the multiset monad $M$ over the powerset monad $P$.

\begin{remark}
The proofs in this section tend to be quite long. To keep the line of reasoning clear, we formulate intermediate results in lemmas and corollaries. The main results are then presented in either theorems or named corollaries.
\end{remark}

\subsection{Multiplicative Zeroes}\label{section:persistentconstants}

Our first focus is on properties \ref{ax:sunits} and \ref{ax:tconst}. The goal is to prove that in a composite of theories~$\bb{S}$ and~$\bb{T}$, the constant $e_\specialopT$ behaves like a multiplicative zero, consuming any $\bb{S}$-term it appears in. More precisely, we will show that any term made from an $\bb{S}$-term with the constant $e_\specialopT$ substituted into one its free variables, is equal to $e_\specialopT$. To get there we first prove the statement for any $\bb{S}$-term made of variables and a single operation symbol. Then we generalise using substitutions and induction.

\begin{lemma}\label{Lemma1}
  Let $\bb{S}$ be an algebraic theory with \refpro~\ref{ax:sunits}, and $\bb{T}$ an algebraic theory with properties \ref{ax:tvar0} and \ref{ax:tconst}. Finally, let $\bb{U}$ be a composite of $\bb{S}$ and $\bb{T}$. Then for every $n$-ary operation ($n \geq 1$) in the signature $\signatureS$ and any index $k \leq n$, the following equation holds in $\bb{U}$:
  \[
  \anyopS_n[e_\specialopT/x_k] \theoryeq{U} e_\specialopT
  \]
  where $e_\specialopT$ is the constant in $\bb{T}$ promised by \ref{ax:tconst}.
\end{lemma}

\begin{proof}
 The statement is trivial in the case that $\anyopS_n$ is a unary operation, so we may assume that $n \geq 2$.
 From the fact that $\bb{U}$ is a composite of the theories $\bb{S}$ and $\bb{T}$, we know that every term in $\bb{U}$ is equal to a $\bb{T}$- term built over $\bb{S}$-terms. And so, there is a $\bb{T}$-term $Y \vdash t$ and there are $\bb{S}$-terms $X_1 \vdash s_1,\ldots, X_m \vdash s_m$ such that:
\begin{align}\label{lambda_1}
\anyopS_n[e_\specialopT/x_k] \theoryeq{U} t[s_j/y_j]
\end{align}

Since we assume \refpro~\ref{ax:sunits}, we know $\anyopS_n$ has unit $e_\anyopS$. Consider the substitution $x_i~\mapsto~e_\anyopS$ for all $i \neq k$. Then:
\begin{align*}
& \anyopS_n[e_\specialopT/x_k] \theoryeq{U} t[s_j/y_j] \\
\Leftrightarrow\;\; & \reason{\text{Substitution:}} \\
& \anyopS_n[e_\specialopT/x_k, e_\anyopS/x_i] \theoryeq{U} t[s_j[e_\anyopS/x_i]/y_j] \\
\Leftrightarrow\;\; & \reason{e_\anyopS \text{is a unit for } \anyopS_n} \\
& e_\specialopT \theoryeq{U} t[s_j[e_\anyopS/x_i]/y_j]
\end{align*}
Because of essential uniqueness (see \refdef~\ref{def:composite}), we conclude that there are functions $\essuniqfunctionA: \emptyset \rightarrow Z, \essuniqfunctionB: Y \rightarrow Z$ such that:
\begin{align*}
& e_\specialopT \theoryeq{T} t[\essuniqfunctionB(y_j)/y_j] \\
\end{align*}
By assumption \ref{ax:tvar0} we may conclude that $t$ does not have any variables and hence we have $t \theoryeq{T} e_\specialopT$.
Going back to \refeq~\eqref{lambda_1}, we see that:
\begin{align*}
& \anyopS_n[e_\specialopT/x_k] \theoryeq{U} t[s_j/y_j] \\
\Leftrightarrow \;\;& \reason{ \var(t) = \emptyset } \\
& \anyopS_n[e_\specialopT/x_k] \theoryeq{U} t \\
\Leftrightarrow \;\; & \reason{ t = e_\specialopT } \\
& \anyopS_n[e_\specialopT/x_k] \theoryeq{U} e_\specialopT
\end{align*}
proving the lemma.
\end{proof}

This immediately leads to a slightly more general statement:

\begin{corollary}\label{Cor1}
   Let $\bb{S}$ and $\bb{T}$ be algebraic theories with properties \ref{ax:sunits}, and \ref{ax:tvar0} and \ref{ax:tconst} respectively, and let $e_\specialopT$ be the constant in $\bb{T}$ promised by \ref{ax:tconst}. Any composite theory $\bb{U}$ of $\bb{S}$ and $\bb{T}$ is such that for every $n$-ary operation $\phi_n$ ($n \geq 1$) in the signature $\signatureS$ and any terms $u_1, \ldots, u_n$ in $\bb{U}$, the following equation holds if there is a $k \leq n$ with $u_k = e_\specialopT$:
  \[
  \anyopS_n[u_i/x_i] \theoryeq{U} e_\specialopT
  \]
\end{corollary}

We use induction to prove that the statement holds not just for operations in the signature, but for any $\bb{S}$-term that has $e_\specialopT$ substituted into one of its free variables.

\newpage
\begin{theorem}\label{Prop1}
  Given algebraic theories $\bb{S}$ and $\bb{T}$ with properties \ref{ax:sunits}, and \ref{ax:tvar0} and \ref{ax:tconst} respectively, let $\bb{U}$ be a composite of $\bb{S}$ and $\bb{T}$. Then for any $\bb{S}$-term $s$ with $n \geq 1$ free variables, and any terms $u_1, \ldots, u_n$ in $\bb{U}$: if there is a $k \leq n$ such that $u_k = e_\specialopT$, then:
  \[
  s[u_i/x_i] \theoryeq{U} e_\specialopT
  \]
\end{theorem}

\begin{proof}
  Proof by induction. The base case is when $s$ is just a variable, in which case the statement is trivially true: $s[u_i/x_i] \theoryeq{U} x[e_\specialopT/x] \theoryeq{U} e_\specialopT$.

  Now suppose that for the terms $s_1,\ldots,s_m$ ($m \geq 1$) our statement holds, and suppose that $\bigcup_{i=1}^m \var(s_i) = \{x_1,\ldots,x_n\}$. Let $\anyopS_m$ be any $m$-ary operation  from the signature $\signatureS$ and consider
  \[
  \anyopS_m[s_j/x_j]
  \]
  We need to show that if $u_1, \ldots, u_n$ are terms in $\bb{U}$ such that there is a $k \leq n$ with $u_k = e_\specialopT$, then:
  \[
  \anyopS_m[s_j[u_i/x_i]/x_j] \theoryeq{U} e_\specialopT
  \]
  To see this, find $1 \leq l \leq m$ such that $s_l$ contains the variable $x_k$, so that we know one of the free variables in $s_l$ is replaced by $e_\specialopT$. By the induction hypothesis, we know that for $s_l$ we have the equality:
  \[
  s_l[u_i/x_i] \theoryeq{U} e_\specialopT
  \]
  Therefore, we know that $\anyopS_m[s_j[u_i/x_i]/x_j]$ satisfies the premiss of \refcor~\ref{Cor1}. We conclude that:
  \[
  \anyopS_m[s_j[u_i/x_i]/x_j] \theoryeq{U} e_\specialopT
  \]
  And hence, by induction, the statement is true for all $\bb{S}$-terms.
\end{proof}

\begin{example}\label{ex:multiplicativezeroes}
  \begin{enumerate}[label={[\arabic*]}]
    \item \emph{List monad}: Since the theory of monoids, corresponding to the list monad, satisfies \ref{ax:sunits}, \ref{ax:tvar0}, and \ref{ax:tconst}, we know that a candidate distributive law for the list monad over itself has to map any list of lists containing the empty list to the empty list. This is the first step in showing that there is no distributive law for the list monad over itself.
    \item \emph{List and lift monad:} Since the theory of pointed sets, corresponding to the lift monad, satisfies \ref{ax:tvar0} and \ref{ax:tconst}, we also know that any distributive law $L(-)_\bot \Rightarrow (-)_\bot L$ has to map every list containing the constant $\bot$ to $\bot$.
    \item \emph{Multiset and powerset monad:} The algebraic theories of commutative monoids and join semilattices, corresponding to the multiset monad and powerset monad respectively, both satisfy all three requirements \ref{ax:sunits}, \ref{ax:tvar0}, and \ref{ax:tconst}. So any distributive law $MP \Rightarrow PM$ has to map every multiset containing the empty set to the empty set. Also, any distributive law $PM \Rightarrow MP$ has to map every set containing the empty multiset to the empty multiset, although we will later see that such a distributive law does not exist.
  \end{enumerate}
\end{example}

If $\bb{T}$ has more than one constant, this could lead to inconsistencies. Our next no-go theorem makes this precise:

\begin{corollary}[No-Go Theorem: Too Many Constants]\label{thm:nogoConstants}
  Let $\bb{S}$ and $\bb{T}$ be algebraic theories with properties \ref{ax:sunits} and \ref{ax:tvar0} respectively. Further assume that there is a term $s$ in $\bb{S}$ with at least two free variables. Then if $\bb{T}$ has more than one constant, there exists no composite theory of $\bb{S}$ and $\bb{T}$.
\end{corollary}

\begin{proof}
  Suppose that $\bb{U}$ is a composite theory of $\bb{S}$ and $\bb{T}$ and let $c_1$ and $c_2$ be distinct constants in $\bb{T}$. Then by \refthm~\ref{Prop1} we have:
  \[
  c_1 \theoryeq{U} s[c_1/x_1, c_2/x_2] \theoryeq{U} c_2
  \]
  Contradiction. So $\bb{U}$ cannot be a composite of $\bb{S}$ and $\bb{T}$.
\end{proof}

By \refcor~\ref{cor:noComposite-noDistlaw}, this also yields a no-go theorem for distributive laws.
We give a few examples of monads that do not distribute over one another as a consequence of this no-go theorem.

\begin{example}[Spotting a mistake in the literature: list and exception monad]
We have seen in \refex~\ref{ex:multiplicativezeroes} that the theory of monoids, corresponding to the list monad, satisfies \ref{ax:sunits}. It also has a term with two free variables, namely $x*y$, where $*$ is the binary operation in the signature of the theory. The exception monad satisfies \ref{ax:tvar0}, so when the exception monad has more than one exception, \refcor~\ref{thm:nogoConstants} states that there is no distributive law $L\circ (-+E) \Rightarrow (-+E) \circ L$.

However, \citet[Example 4.12]{ManesMulry2008} claim to have a distributive law of this type for the case where $E = \{a,b\}$, given by:
\begin{align}
\lambda [] &= [] \\
\lambda [e] & = e \text{ for any exception } e \in E \\
\lambda L & = L \text{ if no element of L is in } E \\
\lambda L & = a \text{ otherwise.}
\end{align}
We check more concretely that this cannot be a distributive law by showing that it fails the first multiplication axiom from \refdef~\ref{def:distlaw}:
\begin{center}
  \begin{tikzcd}[column sep=small]
  {[[b],[]]} \arrow[d, "\mu^L_{EX}"] \arrow[rr, "L(\lambda_X)"] & & {[b, []]} \arrow[rr, "\lambda_{LX}"] & & a \arrow[d, "E(\mu^L_X)"] \\
  {[b]} \arrow[rrr, "\lambda_X"]& & & b \arrow[r, phantom, "\neq"] & a
  \end{tikzcd}
\end{center}
The given distributive law seems to follow directly from \citet[Theorem 4.6]{ManesMulry2008}. This would indicate that in its current form, this theorem produces falsehoods. We leave identifying the scope of the problem with this theorem to later work.

It is important to notice that \refcor~\ref{thm:nogoConstants} does not contradict the well-known result that the exception monad distributes over every set monad $S$; that result is for the other direction $(-+E) \circ S \Rightarrow S \circ (-+E)$.
\end{example}

\begin{example}[Monads that are themselves compositions of other monads]
  \begin{enumerate}[label={[\arabic*]}]
    \item \emph{Ring monad.} We saw in \refex~\ref{ex:ring-monad} how the composition of the list monad and the Abelian group monad resulted in the ring monad. Since the ring monad has two constants, it satisfies the conditions for $T$ in the no-go theorem. Therefore none of the monads list, multiset, or powerset distribute over the ring monad; they all satisfy the conditions for $S$ in the no-go theorem.
    \item A similar story holds for the composition of the multiset monad with itself. Because of \refcor~\ref{thm:nogoConstants}, there is no distributive law $M\circ(M\circ M) \Rightarrow (M\circ M)\circ M$. There could, however, still be a distributive law $(M \circ M) \circ M \Rightarrow M \circ (M\circ M)$.
  \end{enumerate}
\end{example}

\begin{counter}
It is well known that the exception monad distributes over itself. Even though the corresponding theory satisfies properties \ref{ax:sunits} and \ref{ax:tvar0}, there are no terms with more than one free variable, and hence \refcor~\ref{thm:nogoConstants} does not apply.
\end{counter}

\subsection{The One Distributive Law, If It Exists}

Needing just the properties \ref{ax:sunits}, \ref{ax:tvar0}, and \ref{ax:tconst}, \refthm~\ref{Prop1} already greatly restricts the possibilities for a distributive law between monads $S$ and $T$. We will now see that if both $\bb{S}$ and $\bb{T}$ have binary terms with unit (properties \ref{ax:sbinary} and \ref{ax:tunit}), then in a composite theory, the binary of $\bb{S}$ distributes over the binary of $\bb{T}$ like times over plus in \refeq~\eqref{timesoverplus}. For the monads corresponding to these theories, this could mean that there is only one distributive law, if there exists a distributive law at all.

We derive this distributional behaviour in three stages, relying heavily on the separation property of a composite theory: every term can be written as a $\bb{T}$-term of $\bb{S}$-terms. If $\specialopS$ is the binary in $\bb{S}$, and $\specialopT$ the binary in $\bb{T}$, and $t_0[s_i/y_i]$ is the separated term equal to $\specialopS(\specialopT(y_1,y_2), x_0)$, then we derive the following about $t_0[s_i/y_i]$:
\begin{enumerate}
  \item First, we prove which variables appear in the $\bb{S}$-terms of the separated term: either both $y_1$ and $x_0$ or both $y_2$ and $x_0$, and nothing else.
  \item Then, we prove that each of the $\bb{S}$-terms of the separated term is either equal to $\specialopS(y_1,x_0)$ or $\specialopS(y_2,x_0)$.
  \item Finally, we derive that the separated term $t_0[s_i/y_i]$ has to be equal to $\specialopT(\specialopS(y_1,x_0),\specialopS(y_2,x_0))$.
\end{enumerate}

The first step:

\begin{lemma}\label{t_0(s_i)}
  Let $\bb{S}$ and $\bb{T}$ be two algebraic theories satisfying \ref{ax:svar0}, \ref{ax:svar1}, \ref{ax:sunits}, and \ref{ax:sbinary} and \ref{ax:tvar0}, \ref{ax:tconst}, and \ref{ax:tunit} respectively. Let $\bb{U}$ be a composite theory of $\bb{S}$ and $\bb{T}$.  Then there is a $\bb{T}$-term $t_0$ and there are $\bb{S}$-terms $s_i$, $1 \leq i \leq \#\var(t_0)$ such that:
  \[
    \specialopS(\specialopT(y_1,y_2), x_0) \theoryeq{U} t_0[s_i/y'_i]
  \]
  and for each $1 \leq i \leq \#\var(t_0)$:
  \[
  \var(s_i) = \{y_1,x_0\} \qquad \text{or} \qquad \var(s_i) = \{y_2,x_0\}
  \]
  Furthermore, there is an $i$ such that $\var(s_i) = \{y_1,x_0\}$ and an $i$ such that $\var(s_i) = \{y_2,x_0\}$.

  Similarly, there is a $\bb{T}$-term $t_1$ and there are $\bb{S}$-terms $s_j$, $1 \leq j \leq \#\var(t_1)$ such that:
  \[
    \specialopS(x_0,\specialopT(y_1,y_2)) \theoryeq{U} t_1[s_j/y'_j]
  \]
  and for each $1 \leq j \leq \#\var(t_1)$:
  \[
  \var(s_j) = \{y_1,x_0\} \qquad \text{or} \qquad \var(s_j) = \{y_2,x_0\}
  \]
  Furthermore, there is a $j$ such that $\var(s_j) = \{y_1,x_0\}$ and a $j$ such that $\var(s_j) = \{y_2,x_0\}$
\end{lemma}

\begin{proof}
We only explicitly prove the statements for $\specialopS(\specialopT(y_1,y_2), x_0)$. The proof for $\specialopS(x_0, \specialopT(y_1,y_2))$ is similar.

From the fact that $\bb{U}$ is a composite of the theories $\bb{S}$ and $\bb{T}$, we know that every term in $\bb{U}$ is separated. And so, there is a $t_0$ and there are $s_i$, $1 \leq i \leq \#\var(t_0)$ such that:
\begin{align}\label{lambda_2}
\specialopS(\specialopT(y_1,y_2), x_0) \theoryeq{U} t_0[s_i/y'_i]
\end{align}

To prove the second part of the statement, we substitute $x_0 \mapsto e_\specialopS$ in \refeq~\eqref{lambda_2}. This yields:
\begin{align}
& \specialopS(\specialopT(y_1,y_2), e_\specialopS) \theoryeq{U} t_0[s_i[e_\specialopS/x_0]/y'_i] \nonumber \\
\Rightarrow \;\; & \reason{e_\specialopS \text{ is the unit of } \specialopS} \nonumber \\
& \specialopT(y_1,y_2) \theoryeq{U} t_0[s_i[e_\specialopS/x_0]/y'_i] \label{t-knowledge2}
\end{align}
By the essential uniqueness property (see \refdef~\ref{def:composite}), we conclude that there are functions:
\begin{equation*}
  \essuniqfunctionA: \{y_1, y_2\} \rightarrow Z, \essuniqfunctionB: \{y'_i \given 1 \leq i \leq \#\var(t_0)\} \rightarrow Z
\end{equation*}
such that:
\begin{align}\label{t-knowledge1}
& \specialopT[\essuniqfunctionA(y_1)/y_1,\essuniqfunctionA(y_2)/y_2] \theoryeq{T} t_0[\essuniqfunctionB(y'_i)/y'_i]
\end{align}
Furthermore, whenever $\essuniqfunctionA(y_1) = \essuniqfunctionB(y'_i)$ or $\essuniqfunctionA(y_2) = \essuniqfunctionB(y'_i)$, we have respectively:
\begin{align}
  y_1 & \theoryeq{S} s_i[e_\specialopS/x_0] \label{s-knowledge1} \\
  y_2 & \theoryeq{S} s_i[e_\specialopS/x_0] \label{s-knowledge2}
\end{align}
Since we assume variables $y_1$ and $y_2$ to be distinct, essential uniqueness also gives us $\essuniqfunctionA(y_1) \neq \essuniqfunctionA(y_2)$.

We analyse \refeq~\eqref{t-knowledge1} more closely, comparing the variables appearing in both $\specialopT[\essuniqfunctionA(y_1)/y_1,\essuniqfunctionA(y_2)/y_2]$ and $t_0[\essuniqfunctionB(y'_i)/y'_i]$.

Firstly, we show that $\{\essuniqfunctionB(y'_i) \given 1 \leq i \leq \#\var(t_0)\} \subseteq \{\essuniqfunctionA(y_1), \essuniqfunctionA(y_2)\}$. This follows from the following equalities:
\begin{align*}
  & e_\specialopT \\
\theoryeq{T}\; & \reason{e_\specialopT \text{ is the unit for } \specialopT} \\
  & \specialopT(e_\specialopT, e_\specialopT) \\
\theoryeq{T}\; & \reason{\text{Substitution}} \\
  & \specialopT(\essuniqfunctionA(y_1), \essuniqfunctionA(y_2))[e_\specialopT/\essuniqfunctionA(y_1), e_\specialopT/\essuniqfunctionA(y_2)] \\
\theoryeq{T}\; & \reason{\text{Equation }\eqref{t-knowledge1}} \\
  & t_0[\essuniqfunctionB(y'_i)/y'_i][e_\specialopT/\essuniqfunctionA(y_1), e_\specialopT/\essuniqfunctionA(y_2)]
  \end{align*}
So: $e_\specialopT \theoryeq{T} t_0[\essuniqfunctionB(y'_i)/y'_i][e_\specialopT/\essuniqfunctionA(y_1), e_\specialopT/\essuniqfunctionA(y_2)]$. Then by assumption \ref{ax:tvar0}:
\begin{equation*}
  \var(t_0[\essuniqfunctionB(y'_i)/y'_i][e_\specialopT/\essuniqfunctionA(y_1), e_\specialopT/\essuniqfunctionA(y_2)]) = \emptyset
\end{equation*}
Therefore, $t_0[\essuniqfunctionB(y'_i)/y'_i]$ can contain no other variables than $\essuniqfunctionA(y_1)$ and $\essuniqfunctionA(y_2)$. That is:
\begin{equation}\label{eq:variable-subset}
  \{\essuniqfunctionB(y'_i) \given 1 \leq i \leq \#\var(t_0)\} \subseteq \{\essuniqfunctionA(y_1), \essuniqfunctionA(y_2)\}
\end{equation}
Next, we show that both $\essuniqfunctionA(y_1)$ and $\essuniqfunctionA(y_2)$ need to appear in $t_0[\essuniqfunctionB(y'_i)/y'_i]$. Suppose that $\essuniqfunctionA(y_1)$ does not appear in $t_0[\essuniqfunctionB(y'_i)/y'_i]$. Then from \refeq~\eqref{eq:variable-subset} we know that for all $y'_i \in \var(t_0)$, $\essuniqfunctionB(y'_i) = \essuniqfunctionA(y_2)$. Then:
\begin{align*}
 & \essuniqfunctionA(y_1) \\
\theoryeq{T}\; & \reason{e_\specialopT \text{ is the unit for } \specialopT} \\
 & \specialopT(\essuniqfunctionA(y_1), e_\specialopT) \\
\theoryeq{T}\; & \reason{\text{Substitution}} \\
 & \specialopT(\essuniqfunctionA(y_1), \essuniqfunctionA(y_2))[e_\specialopT/\essuniqfunctionA(y_2)] \\
\theoryeq{T}\; & \reason{\text{Equation \eqref{t-knowledge1}}} \\
 & t_0[\essuniqfunctionB(y'_i)/y'_i][e_\specialopT/\essuniqfunctionA(y_2)] \\
\theoryeq{T}\; & \reason{\text{For all } i, \essuniqfunctionB(y'_i) = \essuniqfunctionA(y_2)} \\
 & t_0[e_\specialopT/y'_i]
\end{align*}
So: $\essuniqfunctionA(y_1) \theoryeq{T} t_0[e_\specialopT/y'_i]$, but this contradicts assumption \ref{ax:tvar0}, because $\var(t_0[e_\specialopT/y'_i]) = \emptyset$, since every free variable in $t_0$ has been substituted with the constant $e_\specialopT$, and $\var(\essuniqfunctionA(y_1)) = \{\essuniqfunctionA(y_1)\} \neq \emptyset$.

So $\essuniqfunctionA(y_1)$ has to appear in $t_0[\essuniqfunctionB(y'_i)/y'_i]$. A similar line of reasoning yields the same conclusion for $\essuniqfunctionA(y_2)$. Therefore, there is an $y'_i$ such that $\essuniqfunctionB(y'_i) = \essuniqfunctionA(y_1)$ and there is an $y'_i$ such that $\essuniqfunctionB(y'_i) = \essuniqfunctionA(y_2)$.

In summary, if we define:
\begin{align*}
I_1 = \{i \given \essuniqfunctionB(y'_i) = \essuniqfunctionA(y_1)\} \\
I_2 = \{i \given \essuniqfunctionB(y'_i) = \essuniqfunctionA(y_2)\}
\end{align*}
then we know that neither $I_1$ nor $I_2$ is empty and that $I_1 \cup I_2 = \{1, \ldots, \#\var(t_0)\}$.

We finally consider \refeqs~\eqref{s-knowledge1} and \eqref{s-knowledge2} to reach a conclusion about the variables appearing in the terms $s_i$.

Since for all $i \in I_1: \essuniqfunctionB(y'_i) = \essuniqfunctionA(y_1)$, we have by \refeq~\eqref{s-knowledge1} that $s_i[e_\specialopS/x_0] \theoryeq{S} y_1$. Similarly, for all $i \in I_2$, $s_i[e_\specialopS/x_0] \theoryeq{S} y_2$.
By assumption \ref{ax:svar1}, we conclude that:
\begin{align}
\forall i \in I_1:\; & \{x_0, y_1\} \vdash s_i \label{variable_subset1} \\
\forall i \in I_2:\; & \{x_0, y_2\} \vdash s_i \label{variable_subset2}
\end{align}
In addition, since $y_1 \theoryeq{S} s_i[e_\specialopS/x_0]$, we would have $y_1$ equal to a constant if $y_1$ would not appear in $s_i$ (for $i \in I_1$), contradicting assumption \ref{ax:svar0}. Similarly for $y_2$ and $s_i$ for $i \in I_2$. And so:
\begin{align}
\forall i \in I_1:\; & y_1 \in \var(s_i) \label{y_1_in_var}\\
\forall i \in I_2:\; & y_2 \in \var(s_i) \label{y_2_in_var}
\end{align}.

To prove that $x_0 \in \var(s_i)$ for all $1 \leq i \leq \#\var(t_0)$, we substitute $x_0 \mapsto e_\specialopT$ in \refeq~\eqref{lambda_2}:
\[
\specialopS(\specialopT(y_1,y_2), e_\specialopT) \theoryeq{U} t_0[s_i[e_\specialopT/x_0]/y'_i]
\]
By \refthm~$\ref{Prop1}$, $\specialopS(\specialopT(y_1,y_2), e_\specialopT) \theoryeq{U} e_\specialopT$. Therefore we must have that also $t_0[s_i[e_\specialopT/x_0]/y'_i] \theoryeq{U} e_\specialopT$. Hence by assumption \ref{ax:tvar0}, $\var(t_0[s_i[e_\specialopT/x_0]/y'_i]) = \emptyset$. There are two cases:
\begin{itemize}
\item There is a $1 \leq k \leq \#\var(t_0)$ such that $x_0 \notin \var(s_k)$. Then substituting $x_0 \mapsto e_\specialopT$ has no effect on this $s_k$: $s_k[e_\specialopT/x_0] \theoryeq{U} s_k$. Suppose that $k \in I_1$. Then $y_1 \in \var(s_k)$ and hence: $y_1 \in \var(t_0[s_i[e_\specialopT/x_0]/y'_i])$, contradiction. We get a similar contradiction if $k \in I_2$.
\item All $s_i$ contain the variable $x_0$. Then by \refthm~\ref{Prop1}:
  \begin{equation*}
    s_i[e_\specialopT/x_0] \theoryeq{U} e_\specialopT
  \end{equation*}
  Then $t_0[s_i[e_\specialopT/x_0]/y'_i] \theoryeq{U} t_0[e_\specialopT/y'_i]$, agreeing with the fact that $\var(t_0[s_i[e_\specialopT/x_0]/y'_i]) = \emptyset$.
\end{itemize}
And so, for all $1 \leq i \leq \#\var(t_0)$, $x_0 \in \var(s_i)$.

This, together with \refeqs~\eqref{variable_subset1}, \eqref{variable_subset2}, \eqref{y_1_in_var} and \eqref{y_2_in_var}, proves that:
\[
\var(s_i) = \{y_1,x_0\} \qquad \text{or} \qquad \var(s_i) = \{y_2,x_0\}
\]
The fact that neither $I_1$ nor $I_2$ are empty means that this proves the lemma.
\end{proof}

With the first step done, we move on to the second step:

\begin{lemma}\label{t_0-diamond}
  Let $\bb{S}$ and $\bb{T}$ be two algebraic theories satisfying \ref{ax:svar0}, \ref{ax:svar1}, \ref{ax:sunits}, and \ref{ax:sbinary} and \ref{ax:tvar0}, \ref{ax:tvar1}, \ref{ax:tconst}, and \ref{ax:tunit} respectively. Let $\bb{U}$ be a composite theory of $\bb{S}$ and $\bb{T}$.  Then there is a $\bb{T}$-term $t_0$ and there are $\bb{S}$-terms $s_i$, $1 \leq i \leq \#\var(t_0)$ such that:
  \[
    \specialopS(\specialopT(y_1,y_2), x_0) \theoryeq{U} t_0[s_i/y'_i]
  \]
  and for each $1 \leq i \leq \#\var(t_0)$:
  \[
  s_i \theoryeq{S} \specialopS(y_1,x_0) \qquad \text{or} \qquad s_i \theoryeq{S} \specialopS(y_2,x_0)
  \]
  Furthermore, there is an $i$ such that $s_i \theoryeq{S} \specialopS(y_1,x_0)$ and there is an $i$ such that $s_i \theoryeq{S} \specialopS(y_2,x_0)$.

  Similarly, there is a $\bb{T}$-term $t_1$ and there are $\bb{S}$-terms $s_j$, $1 \leq j \leq \#\var(t_1)$ such that:
  \[
    \specialopS(x_0,\specialopT(y_1,y_2)) \theoryeq{U} t_1[s_j/y'_j]
  \]
  and for each $1 \leq j \leq \#\var(t_1)$:
  \[
  s_j \theoryeq{S} \specialopS(x_0,y_1) \qquad \text{or} \qquad s_j \theoryeq{S} \specialopS(x_0,y_2)
  \]
  Furthermore, there is a $j$ such that $s_j \theoryeq{S} \specialopS(y_1,x_0)$ and there is an $j$ such that $s_j \theoryeq{S} \specialopS(y_2,x_0)$.
\end{lemma}

\begin{proof}
Again, we only explicitly prove the statements for $\specialopS(\specialopT(y_1,y_2), x_0)$. The second half of the claim follows similarly.

From \reflem~\ref{t_0(s_i)} we know that there is a $\bb{T}$-term $t_0$ and there are $\bb{S}$-terms $s_i$ with
\begin{equation*}
  1 \leq i \leq \#\var(t_0)
\end{equation*}
such that:
  \[
    \specialopS(\specialopT(y_1,y_2), x_0) \theoryeq{U} t_0[s_i/y'_i]
  \]
We substitute $y_1 \mapsto e_\specialopT$:
\begin{align}
  & \specialopS(\specialopT(e_\specialopT,y_2), x_0) \theoryeq{U} t_0[s_i[e_\specialopT/y_1]/y'_i] \nonumber\\
  \Rightarrow \;\; & \reason{e_\specialopT \text{ is the unit of } \specialopT} \nonumber\\
  & \specialopS(y_2,x_0) \theoryeq{U} t_0[s_i[e_\specialopT/y_1]/y'_i] \nonumber\\
  \Rightarrow \;\; & \reason{\text{Substitution}} \nonumber\\
  & z[\specialopS(y_2,x_0)/z] \theoryeq{U} t_0[s_i[e_\specialopT/y_1]/y'_i] \label{eq:z-t0}
\end{align}
Where $z$ is just a variable, so that $z[\specialopS(y_2,x_0)/z]$ is a $\bb{T}$-term built out of $\bb{S}$-terms.

To use essential uniqueness from \refdef~\ref{def:composite}, we need two terms written as $\bb{T}$-terms of $\bb{S}$-terms. However, $t_0[s_i[e_\specialopT/y_1]/y'_i]$ is a $\bb{T}$-term built out of $\bb{S}$-terms with possibly a $\bb{T}$-constant in them. So we need to write $t_0[s_i[e_\specialopT/y_1]/y'_i]$ into a form where it is just a $\bb{T}$-term built out of $\bb{S}$-terms. We use \refthm~\ref{Prop1} in combination with our knowledge from \reflem~\ref{t_0(s_i)} about the variables appearing in each $s_i$ to do just that. Define:
\begin{align*}
I_1 & = \{i \given 1 \leq i \leq \#\var(t_0), \var(s_i) = \{y_1, x_0\}\} \\
I_2 & = \{i \given 1 \leq i \leq \#\var(t_0), \var(s_i) = \{y_2,x_0\}\}
\end{align*}
From \reflem~\ref{t_0(s_i)} we know that neither $I_1$ nor $I_2$ is empty and their union contains all $1 \leq i \leq \#\var(t_0)$. For each $i \in I_1$, we use \refthm~\ref{Prop1} to rewrite $s_i[e_\specialopT/y_1] \theoryeq{U} e_\specialopT$. For each $i \in I_2$, we know that since $y_1$ does not appear in $s_i$, $s_i[e_\specialopT/y_1] \theoryeq{U} s_i$. Therefore:
\[
t_0[s_i[e_\specialopT/y_1]/y'_i] \theoryeq{T} t_0[e_\specialopT/y'_i (i \in I_1), s_i/y'_i (i \in I_2)]
\]
Next, define:
\[
t'_0 = t_0[e_\specialopT/y'_i (i \in I_1)]
\]
Then $t'_0$ is a $\bb{T}$-term with free variables $\{y'_i \given i \in I_2\}$. We have:
\begin{align*}
   & t'_0[s_i/y'_i] \\
  \theoryeq{T} \;  & \reason{\text{Definition of } t'_0} \\
   & t_0[e_\specialopT/y'_i (i \in I_1), s_i/y'_i (i \in I_2)]  \\
  \theoryeq{T} \; &  t_0[s_i[e_\specialopT/y_1]/y'_i] \\
  \theoryeq{T} \; & \reason{\text{Equation \eqref{eq:z-t0}}} \\
   &  z[\specialopS(y_2,x_0)/z]
\end{align*}
Now we can use the essential uniqueness property, and conclude that there are functions $\essuniqfunctionA: \{z\} \rightarrow Z, \essuniqfunctionB: \{y'_i \given i \in I_2\} \rightarrow Z$ such that:
\begin{align}\label{t-knowledge3}
& z[\essuniqfunctionA(z)/z] \theoryeq{T} t'_0[\essuniqfunctionB(y'_i)/y'_i]
\end{align}
Furthermore, whenever $\essuniqfunctionB(y'_i) = \essuniqfunctionA(z)$, we have:
\begin{align}
   s_i & \theoryeq{S} \specialopS(y_2,x_0) \label{s-knowledge3}
\end{align}

From \refeq~\eqref{t-knowledge3} and assumption \ref{ax:tvar1} we conclude that $\{\essuniqfunctionA(z)\} \vdash t'_0[\essuniqfunctionB(y'_i)/y'_i]$. And hence for all $i \in I_2$, $\essuniqfunctionB(y'_i) = \essuniqfunctionA(z)$. So by \refeq~\eqref{s-knowledge3}, for each $i \in I_2$: $s_i \theoryeq{S} \specialopS(y_2,x_0)$, which gives us half of the desired conclusion about each $s_i$ ($1 \leq i \leq \#\var(t_0)$).

Finally, a similar argument using the substitution $y_2 \mapsto e_\specialopT$ instead of $y_1 \mapsto e_\specialopT$ leads to the conclusion that for each $i \in I_1$, $s_i \theoryeq{S} \specialopS(y_1, x_0)$.
\end{proof}

And finally, the last step, proving that $\specialopS$ distributes over $\specialopT$:

\begin{theorem}[Times over Plus Theorem]\label{theorem}
Let $\bb{S}$ and $\bb{T}$ be two algebraic theories satisfying \ref{ax:svar0}, \ref{ax:svar1}, \ref{ax:sunits}, and \ref{ax:sbinary} and \ref{ax:tvar0}, \ref{ax:tvar1}, \ref{ax:tconst}, and \ref{ax:tunit} respectively. Let $\bb{U}$ be a composite theory of $\bb{S}$ and $\bb{T}$.  Then $\specialopS$ distributes over $\specialopT$:
\begin{align}
    \specialopS(\specialopT(y_1,y_2), x_0) & \theoryeq{U} \specialopT(\specialopS(y_1,x_0),\specialopS(y_2,x_0)) \label{theoremeq1} \\
    \specialopS(x_0, \specialopT(y_1,y_2)) & \theoryeq{U} \specialopT(\specialopS(x_0,y_1),\specialopS(x_0,y_2)) \label{theoremeq2}
\end{align}
\end{theorem}

\begin{proof}
Again, we only explicitly prove the first statement, as the proof of the second statement is similar.

From \reflem~\ref{t_0-diamond} we know that there is a $\bb{T}$-term $t_0$ and there are $\bb{S}$-terms $s_i$ with:
\begin{equation*}
  1 \leq i \leq \#\var(t_0)
\end{equation*}
such that either $s_i = \specialopS(y_1,x_0)$ or $s_i = \specialopS(y_2,x_0)$ and:
 \[
   \specialopS(\specialopT(y_1,y_2), x_0) \theoryeq{U} t_0[s_i/y'_i]
 \]
Define:
\begin{align*}
I_1 & = \{i \given 1 \leq i \leq \#\var(t_0), s_i = \specialopS(y_1,x_0)\} \\
I_2 & = \{i \given 1 \leq i \leq \#\var(t_0), s_i = \specialopS(y_2,x_0)\}
\end{align*}
Then, using the substitution $x_0 \mapsto e_\specialopS$, we get:
\begin{align*}
& \specialopT(y_1,y_2) \\
\theoryeq{U} \; & \reason{e_\specialopS \text{ is the unit of } \specialopS} \\
& \specialopS(\specialopT(y_1,y_2), e_\specialopS) \\
\theoryeq{U} \; & \reason{\text{Substitution}} \\
& \specialopS(\specialopT(y_1,y_2), x_0)[e_\specialopS/x_0] \\
\theoryeq{U} \; & \reason{\text{Lemma \ref{t_0-diamond}}} \\
& t_0[s_i[e_\specialopS/x_0]/y'_i] \\
\theoryeq{U} \; & \reason{s_i = \specialopS(y_1,x_0) \text{ or } s_i = \specialopS(y_2,x_0) \text{ and } e_\specialopS \text{ is the unit of } \specialopS} \\
& t_0 [y_1/y'_i (i \in I_1), y_2/y'_i (i \in I_2)]
\end{align*}
So:
\begin{equation}\label{eq:theorem1}
  \specialopT(y_1,y_2) \theoryeq{U} t_0 [y_1/y'_i (i \in I_1), y_2/y'_i (i \in I_2)]
\end{equation}
We also have:
\begin{align*}
& t_0[s_i/y'_i]\\
\theoryeq{U} \; & \reason{\text{Lemma \ref{t_0-diamond}}} \\
& t_0[\specialopS(y_1,x_0)/y'_i (i \in I_1), \specialopS(y_2,x_0)/y'_i (i \in I_2)] \\
\theoryeq{U}\; & \reason{\text{Substitution: } \text{term} = y[\text{term}/y]} \\
& t_0[y_1[\specialopS(y_1,x_0)/y_1]/y'_i (i \in I_1), y_2[\specialopS(y_2,x_0)/y_2]/y'_i (i \in I_2)]
\end{align*}
So we conclude:
\begin{align*}
& t_0[s_i/y'_i] \\
\theoryeq{U} \; &  t_0[y_1[\specialopS(y_1,x_0)/y_1]/y'_i (i \in I_1), y_2[\specialopS(y_2,x_0)/y_2]/y'_i (i \in I_2)] \\
\theoryeq{U} \; & \reason{\text{Equation \eqref{eq:theorem1}}} \\
&  \specialopT(y_1,y_2)[\specialopS(y_1,x_0)/y_1, \specialopS(y_2,x_0)/y_2] \\
\theoryeq{U} \; &  \specialopT(\specialopS(y_1,x_0),\specialopS(y_2,x_0))
\end{align*}
Which proves the theorem.
\end{proof}

This greatly reduces the search space for distributive laws, and in some cases even narrows it down to precisely one possibility:

\begin{corollary}[Upper Limit]\label{cor:upperlimit}
   Let $S$ and $T$ be two monads with corresponding algebraic theories $\bb{S}$ and $\bb{T}$ that have signatures with one constant and one binary operation, and nothing else. If for both theories the constant acts as a unit for the binary operation and the theories further satisfy \ref{ax:svar0} and \ref{ax:svar1}, and \ref{ax:tvar0} and \ref{ax:tvar1} respectively, then there exists \emph{at most} one distributive law $ST \Rightarrow TS$.
\end{corollary}

\begin{example}[List Monad]
  Corollary \ref{cor:upperlimit} applies to the list monad, and hence there is at most one distributive law for this list monad over itself: the one that distributes lists of lists like times over plus.
\end{example}

\begin{example}[Unique Distributive Laws]\label{ex:uniquedistlaws}
 Let $S$ be any of the monads tree, list, or multiset. Then the corresponding algebraic theory $\bb{S}$ contains only linear equations. Let $T$ be either the multiset or powerset monad. Since the multiset and powerset monads are commutative monads, it follows from \citet[Theorem 4.3.4]{ManesMulry2007} that there is a distributive law $ST \Rightarrow TS$. Corollary \ref{cor:upperlimit} states that this distributive law is unique. In particular, the distributive law for the multiset monad over itself mentioned in \refex~\ref{ex:multisetdistlaw} is unique.
\end{example}

\subsection{Lacking the Abides Property: a No-Go Theorem}

With \refthm~\ref{theorem} narrowing down the possible distributive laws for two monads, it is easier to find cases in which no distributive law can exist at all. We found two properties that clash with \refthm~\ref{theorem}, one for $\bb{T}$ and one for $\bb{S}$. In this section we show that not satisfying the abides equation, \refpro~\ref{ax:tspecialproperty}, in combination with \refthm~\ref{theorem} prevents the existence of a distributive law. In the next section we do the same for idempotence, \refpro~\ref{ax:sidem}. Both properties are quite common, making the resulting no-go theorems hold for significant classes of monads.

\begin{theorem}[No-Go Theorem: Lacking Abides]\label{thm:nogoTreeList}
  If $\bb{S}$ and $\bb{T}$ are algebraic theories satisfying the conditions of \refthm~\ref{theorem} and $\bb{T}$ additionally satisfies \ref{ax:tspecialproperty}, then there does not exist a composite theory of $\bb{S}$ and $\bb{T}$.
\end{theorem}

\newpage
\begin{proof}
Suppose there exists a composite theory $\bb{U}$. Given \refthm~\ref{theorem}, we compute a separated term equal in $\bb{U}$ to $\specialopS(\specialopT(y_1,y_2), \specialopT(y_3,y_4))$:

\begin{align}
   & \specialopS(\specialopT(y_1,y_2), \specialopT(y_3,y_4)) \nonumber\\
\theoryeq{U}\; & \reason{\text{Substitution}} \nonumber\\
   & \specialopS(\specialopT(y_1,y_2),x_0)[\specialopT(y_3,y_4)/x_0] \nonumber\\
\theoryeq{U}\; & \reason{\text{Equation \eqref{theoremeq1} from \refthm~\ref{theorem}}} \nonumber\\
   & \specialopT(\specialopS(y_1,x_0), \specialopS(y_2,x_0))[\specialopT(y_3,y_4)/x_0] \nonumber\\
\theoryeq{U}\; & \reason{\text{Substitution}} \nonumber\\
   & \specialopT(\specialopS(y_1,\specialopT(y_3,y_4)), \specialopS(y_2,\specialopT(y_3,y_4))) \nonumber\\
\theoryeq{U}\; & \reason{\text{Equation \eqref{theoremeq2} from \refthm~\ref{theorem}}} \nonumber\\
   & \specialopT(\specialopT(\specialopS(y_1,y_3),\specialopS(y_1,y_4)), \specialopT(\specialopS(y_2,y_3),\specialopS(y_2,y_4))) \label{eq:nogoTreeList1}
\end{align}
Notice that we made a choice, taking out the right $\specialopT$ term in
\begin{equation*}
  \specialopS(\specialopT(y_1,y_2), \specialopT(y_3,y_4)) \theoryeq{U} \specialopS(\specialopT(y_1,y_2),x_0)[\specialopT(y_3,y_4)/x_0]
\end{equation*}
rather than the left:
\begin{equation*}
  \specialopS(\specialopT(y_1,y_2), \specialopT(y_3,y_4)) \theoryeq{U} \specialopS(x_0, \specialopT(y_3,y_4))[\specialopT(y_1,y_2)/x_0]
\end{equation*}
The latter option yields:
\begin{align}
    & \specialopS(\specialopT(y_1,y_2), \specialopT(y_3,y_4)) \nonumber \\
\theoryeq{U}\; & \reason{\text{Substitution}} \nonumber\\
    & \specialopS(x_0, \specialopT(y_3,y_4))[\specialopT(y_1,y_2)/x_0] \nonumber\\
\theoryeq{U}\; & \reason{\text{Equation \eqref{theoremeq2} from \refthm~\ref{theorem}}} \nonumber\\
    & \specialopT(\specialopS(x_0,y_3), \specialopS(x_0,y_4))[\specialopT(y_1,y_2)/x_0] \nonumber\\
\theoryeq{U}\; & \reason{\text{Substitution}} \nonumber \\
    & \specialopT(\specialopS(\specialopT(y_1,y_2),y_3), \specialopS(\specialopT(y_1,y_2), y_4)) \nonumber\\
\theoryeq{U}\; & \reason{\text{Equation \eqref{theoremeq1} from \refthm~\ref{theorem}}} \nonumber\\
 & \specialopT(\specialopT(\specialopS(y_1,y_3),\specialopS(y_2,y_3)),\specialopT(\specialopS(y_1,y_4),\specialopS(y_2,y_4)) \label{eq:nogoTreeList2}
\end{align}
Of course, both computations are equally valid, so the terms in \refeqs~\eqref{eq:nogoTreeList1} and \eqref{eq:nogoTreeList2} must be equal:
\begin{align*}
& \specialopT(\specialopT(\specialopS(y_1,y_3),\specialopS(y_1,y_4)),\specialopT(\specialopS(y_2,y_3),\specialopS(y_2,y_4))) \\ \theoryeq{U}\; & \specialopT(\specialopT(\specialopS(y_1,y_3),\specialopS(y_2,y_3)),\specialopT(\specialopS(y_1,y_4),\specialopS(y_2,y_4)))
\end{align*}
Since these are two separated terms that are equal, we can apply the essential uniqueness property, stating that there exist functions:
\begin{align*}
  \essuniqfunctionA: \{y'_1,y'_2,y'_3,y'_4\} \rightarrow Z\\
  \essuniqfunctionB: \{y'_1,y'_2,y'_3,y'_4\} \rightarrow Z
\end{align*}
such that:
\begin{itemize}
  \item Equality in $\bb{T}$:
         \begin{align*}
         & \specialopT(\specialopT(\essuniqfunctionA(y'_1),\essuniqfunctionA(y'_2)), \specialopT(\essuniqfunctionA(y'_3),\essuniqfunctionA(y'_4))) \\ \theoryeq{T}\; &  \specialopT(\specialopT(\essuniqfunctionB(y'_1),\essuniqfunctionB(y'_2)), \specialopT(\essuniqfunctionB(y'_3),\essuniqfunctionB(y'_4)))
         \end{align*}
  \item $\essuniqfunctionA(y'_i) = \essuniqfunctionB(y'_j)$ iff the $\bb{S}$-terms substituted for $y'_i$ and $y'_j$ in \refeqs~\eqref{eq:nogoTreeList1} and~\eqref{eq:nogoTreeList2} are equal.
\end{itemize}
From the second part of essential uniqueness we get that:
\begin{align*}
\essuniqfunctionA(y'_1) & = \essuniqfunctionB(y'_1) & \essuniqfunctionA(y'_3) & = \essuniqfunctionB(y'_2)\\
\essuniqfunctionA(y'_2) & = \essuniqfunctionB(y'_3) & \essuniqfunctionA(y'_4) & = \essuniqfunctionB(y'_4)
\end{align*}
Putting this in the first part then yields:
\begin{align*}
& \specialopT(\specialopT(\essuniqfunctionA(y'_1),\essuniqfunctionA(y'_2)), \specialopT(\essuniqfunctionA(y'_3),\essuniqfunctionA(y'_4))) \\ \theoryeq{T}\; & \specialopT(\specialopT(\essuniqfunctionA(y'_1),\essuniqfunctionA(y'_3)), \specialopT(\essuniqfunctionA(y'_2),\essuniqfunctionA(y'_4)))
\end{align*}
And so by \refpro~\ref{ax:tspecialproperty}:  $\#\var(\specialopT(\specialopT(\essuniqfunctionA(y'_1),\essuniqfunctionA(y'_2)), \specialopT(\essuniqfunctionA(y'_3),\essuniqfunctionA(y'_4)))) \leq 3$. So there must be $i,j$ such that $i \neq j$ and $\essuniqfunctionA(y'_i) = \essuniqfunctionA(y'_j)$. Suppose without loss of generality that $\essuniqfunctionA(y'_1) = \essuniqfunctionA(y'_2)$. Then we must have that $\specialopS(y_1,y_3) \theoryeq{S} \specialopS(y_1,y_4)$. But then also:
\begin{align*}
   & y_1 \\
   \theoryeq{S} \; & \reason{ e_\specialopS \text{ is a unit for } \specialopS } \\
   & \specialopS(y_1, e_\specialopS) \\
   \theoryeq{S} \; & \reason{\text{Substitution}} \\
   &  \specialopS(y_1, y_3)[e_\specialopS/y_3] \\
   \theoryeq{S} \; & \reason{\specialopS(y_1, y_3) = \specialopS(y_1, y_4)} \\
   & \specialopS(y_1, y_4)[e_\specialopS/y_3] \\
   \theoryeq{S} \; & \reason{\text{Substitution: no } y_3 \text{ in } \specialopS(y_1, y_4) } \\
   & \specialopS(y_1, y_4)
\end{align*}

But then, by \refpro~\ref{ax:svar1}: $\{y_1\} \vdash \specialopS(y_1, y_4)$, and so we must have $y_4 = y_1$, contradiction.
The same argument holds for any other $i,j$ pair. Therefore, the existence of a composite theory leads to a contradiction. In other words, no such composite theory exists.
\end{proof}

\begin{corollary}
 By \refcor~\ref{cor:noComposite-noDistlaw}, if the algebraic theories $\bb{S}$ and $\bb{T}$ corresponding to monads $S$ and $T$ satisfy \ref{ax:svar0}, \ref{ax:svar1}, \ref{ax:sunits}, and \ref{ax:sbinary} and \ref{ax:tvar0}, \ref{ax:tvar1}, \ref{ax:tconst}, \ref{ax:tunit}, and \ref{ax:tspecialproperty} respectively, then there does not exist a distributive law $ST \Rightarrow TS$.
\end{corollary}

\begin{example}[No Distributive Law for the List Monad over Itself]
  This finally settles the question of whether the list monad distributes over itself, posed by \citet{ManesMulry2007,ManesMulry2008}. As mentioned in \refex~\ref{ex:STproperties}, the algebraic theory corresponding to the list monad satisfies all of \ref{ax:svar0}, \ref{ax:svar1}, \ref{ax:sunits}, \ref{ax:sbinary}, \ref{ax:tvar0}, \ref{ax:tvar1}, \ref{ax:tconst}, \ref{ax:tunit}, and \ref{ax:tspecialproperty}. Hence there is no distributive law for the list monad over itself.
\end{example}

\begin{remark}
  Although there is no distributive law for the list monad over itself, the functor $LL$ does still carry a monad structure. We are very grateful to Bartek Klin for pointing this out to us. The monad structure on $LL$ can be described as follows:
  \begin{itemize}
    \item There is a distributive law for the list monad over the non-empty list monad $LL^+ \Rightarrow L^+L$ \citep{ManesMulry2007}.
    \item There is a distributive law for the resulting monad over the lift monad $(L^+L)(-)_\bot \Rightarrow (-)_\bot (L^+L)$, derived from general principles \citep{ManesMulry2007}.
    \item The resulting functor $(-)_\bot (L^+L)$ is isomorphic to $LL$, and carries a monad structure. Hence $LL$ carries a monad structure, but not one that can be derived from a distributive law $LL \Rightarrow LL$.
  \end{itemize}
\end{remark}

\begin{example}[No Distributive Law of the Binary Tree Monad over Itself]
The algebraic theory corresponding to the binary tree monad satisfies all of \ref{ax:svar0}, \ref{ax:svar1}, \ref{ax:sunits}, \ref{ax:sbinary}, \ref{ax:tvar0}, \ref{ax:tvar1}, \ref{ax:tconst}, \ref{ax:tunit}, and \ref{ax:tspecialproperty}. And so, like the list monad, there is no distributive law for the binary tree monad over itself.
\end{example}

\begin{example}[Ternary and $n$-ary Trees]
The algebraic theory corresponding to ternary trees has a signature consisting of one constant $e$ and one ternary operation $\phi$, with the following unit equations:
\begin{align*}
\phi(x,e,e) & = x & \phi(e,x,e) & = x & \phi(e,e,x) & = x
\end{align*}
These trees satisfy all of \ref{ax:svar0}, \ref{ax:svar1}, \ref{ax:sunits}, \ref{ax:sbinary}, \ref{ax:tvar0}, \ref{ax:tvar1}, \ref{ax:tconst}, \ref{ax:tunit}, and \ref{ax:tspecialproperty}. To see this, take for properties requiring a binary term the term $t(x,y)$, defined as:
\[
t(x,y) = \phi(x,y,e)
\]
Then $t(x,e) = \phi(x,e,e) = x$ and $t(e,x) = \phi(e,x,e) = x$ so $t$ satisfies \ref{ax:sbinary} and \ref{ax:tunit}, and since there are no other equations, $t$ also satisfies \ref{ax:tspecialproperty}. And so by \refthm~\ref{thm:nogoTreeList} there is no distributive law for the ternary tree monad over itself.

This idea generalises to $n$-ary trees, meaning that for these too, there is no distributive law. In fact, there is no distributive law for any $n$-ary tree monad over any $m$-ary tree monad $(m,n \geq 2)$.
\end{example}

There are many more examples for \refthm~\ref{thm:nogoTreeList}, some of which will be explored in \refsec~\ref{section:boom}. For now, we illustrate the importance of \refpro~\ref{ax:tspecialproperty} with a monad that has all properties required by \refthm~\ref{thm:nogoTreeList}, except for \ref{ax:tspecialproperty}.

\begin{counter}[Multiset Monad]\label{counter:multisetNolist}
The multiset monad is closely related to the list monad, with an algebraic theory having just one extra equation compared to the list monad: commutativity. Because of this equation, the theory does not have \refpro~\ref{ax:tspecialproperty}, but it still has properties \ref{ax:svar0}, \ref{ax:svar1}, \ref{ax:sunits}, \ref{ax:sbinary}, \ref{ax:tvar0}, \ref{ax:tvar1}, \ref{ax:tconst}, and \ref{ax:tunit}. As we have seen in \refex~\ref{ex:uniquedistlaws}, there is a unique distributive law for the multiset monad over itself.
\end{counter}

\newpage
\subsection{Yet Another No-Go Theorem Caused by Idempotence}
In \refsec~\ref{Plotkin} we learned that having idempotent terms in both algebraic theories could prevent the existence of a distributive law. We will now see that together with \refthm~\ref{theorem}, having an idempotent term in $\bb{S}$ is enough. This fills in the remaining unknown combination of the multiset and powerset monad: there is no distributive law $PM \Rightarrow MP$.
\begin{theorem}[No-Go Theorem: Idempotence and Units] \label{thm:nogoTreeIdem}
  Let $\bb{S}$ and $\bb{T}$ be algebraic theories satisfying \ref{ax:svar0}, \ref{ax:svar1}, \ref{ax:sunits}, \ref{ax:sbinary}, and \ref{ax:sidem} and \ref{ax:tvar0}, \ref{ax:tvar1}, \ref{ax:tconst}, and \ref{ax:tunit} respectively. Then there exists no composite theory of $\bb{S}$ and $\bb{T}$.
\end{theorem}
\begin{proof}
  Suppose such a composite theory $\bb{U}$ exists. Then we have:
  \begin{align*}
     & \specialopT(y_1,y_2) \\
  \theoryeq{U}\; & \reason{\text{\ref{ax:sidem}: } \specialopS \text{ is idempotent} } \\
     & \specialopS(\specialopT(y_1,y_2), \specialopT(y_1,y_2)) \\
  \theoryeq{U}\; & \reason{\text{Substitution}} \\
     & \specialopS(\specialopT(y_1,y_2),x_0)[\specialopT(y_1,y_2)/x_0] \\
  \theoryeq{U}\; & \reason{\text{Equation \eqref{theoremeq1} from \refthm~\ref{theorem}}} \\
     & \specialopT(\specialopS(y_1,x_0), \specialopS(y_2,x_0))[\specialopT(y_1,y_2)/x_0] \\
  \theoryeq{U}\; & \reason{\text{Substitution}} \\
     & \specialopT(\specialopS(y_1,\specialopT(y_1,y_2)), \specialopS(y_2,\specialopT(y_1,y_2))) \\
  \theoryeq{U}\; & \reason{\text{Equation \eqref{theoremeq2} from \refthm~\ref{theorem}}} \\
     & \specialopT(\specialopT(\specialopS(y_1,y_1),\specialopS(y_1,y_2)), \specialopT(\specialopS(y_2,y_1),\specialopS(y_2,y_2))) \\
  \theoryeq{U}\; & \reason{\text{\ref{ax:sidem}: } \specialopS \text{ is idempotent} } \\
     & \specialopT(\specialopT(y_1,\specialopS(y_1,y_2)), \specialopT(\specialopS(y_2,y_1),y_2))
  \end{align*}

  So from the essential uniqueness property, we may conclude that there are functions $\essuniqfunctionA: \{y_1,y_2\} \mapsto Z, \essuniqfunctionB: \{y'_1,y'_2,y'_3,y'_4\} \mapsto Z$ such that:
  \begin{align*}
    \specialopT(\essuniqfunctionA(y_1), \essuniqfunctionA(y_2)) & \theoryeq{T} \specialopT(\specialopT(\essuniqfunctionB(y'_1),\essuniqfunctionB(y'_2)), \specialopT(\essuniqfunctionB(y'_3), \essuniqfunctionB(y'_4)))
  \end{align*}
  And $\essuniqfunctionA(y_i) = \essuniqfunctionB(y'_j)$ if and only if the $\bb{S}$-terms substituted for $y_i$ and $y'_j$ in $\specialopT(y_1,y_2)$ and $\specialopT(\specialopT(y_1,\specialopS(y_1,y_2)), \specialopT(\specialopS(y_2,y_1),y_2))$ are equal. From this we immediately get:
  \begin{align*}
   \essuniqfunctionA(y_1) & = \essuniqfunctionB(y'_1) \\
   \essuniqfunctionA(y_2) & = \essuniqfunctionB(y'_4)
  \end{align*}
   Also, from \refpro~\ref{ax:tvar1} we know $\essuniqfunctionA(y_1) \neq \essuniqfunctionA(y_2)$.

   We show that $\{\essuniqfunctionB(y'_2), \essuniqfunctionB(y'_3)\} \subseteq \{\essuniqfunctionA(y_1),\essuniqfunctionA(y_2)\}$. Since we have:
   \[
   \specialopT(\essuniqfunctionA(y_1), \essuniqfunctionA(y_2)) \theoryeq{T} \specialopT(\specialopT(\essuniqfunctionB(y'_1), \essuniqfunctionB(y'_2)), \specialopT(\essuniqfunctionB(y'_3),\essuniqfunctionB(y'_4)))
   \]
   We also have:
   \begin{align*}
   & \specialopT(\essuniqfunctionA(y_1), \essuniqfunctionA(y_2))[e_\specialopT/\essuniqfunctionA(y_1),e_\specialopT/\essuniqfunctionA(y_2)] \\
   & \theoryeq{T} \specialopT(\specialopT(\essuniqfunctionB(y'_1), \essuniqfunctionB(y'_2)), \specialopT(\essuniqfunctionB(y'_3),\essuniqfunctionB(y'_4)))[e_\specialopT/\essuniqfunctionA(y_1),e_\specialopT/\essuniqfunctionA(y_2)] \\
   \Rightarrow \; & \reason{\text{Substitution, and } \essuniqfunctionA(y_1) = \essuniqfunctionB(y'_1), \essuniqfunctionA(y_2) = \essuniqfunctionB(y'_4)} \\
   & \specialopT(e_\specialopT, e_\specialopT) \theoryeq{T} \specialopT(\specialopT(e_\specialopT, \essuniqfunctionB(y'_2)), \specialopT(\essuniqfunctionB(y'_3),e_\specialopT))[e_\specialopT/\essuniqfunctionA(y_1),e_\specialopT/\essuniqfunctionA(y_2)] \\
   \Rightarrow \; & \reason{\text{\ref{ax:tunit}: }e_\specialopT \text{ is the unit of } \specialopT} \\
   & e_\specialopT \theoryeq{T} \specialopT(\essuniqfunctionB(y'_2),\essuniqfunctionB(y'_3))[e_\specialopT/\essuniqfunctionA(y_1),e_\specialopT/\essuniqfunctionA(y_2)]
   \end{align*}

   So by \ref{ax:tvar0}, $\var(\specialopT(\essuniqfunctionB(y'_2), \essuniqfunctionB(y'_3))[e_\specialopT/\essuniqfunctionA(y_1),e_\specialopT/\essuniqfunctionA(y_2)]) = \emptyset$. So we must have:
   \begin{equation*}
     \{\essuniqfunctionB(y'_2), \essuniqfunctionB(y'_3)\} \subseteq \{\essuniqfunctionA(y_1),\essuniqfunctionA(y_2)\}
   \end{equation*}
   But then, by the second part of the essential uniqueness property:
   \begin{equation*}
     \specialopS(y_1,y_2) \theoryeq{S} y_1 \quad \text{or} \quad \specialopS(y_1,y_2) \theoryeq{S} y_2
   \end{equation*}
   Both contradict \ref{ax:svar1}. Therefore, the composite theory $\bb{U}$ cannot exist.
\end{proof}

\begin{corollary}
  By \refcor~\ref{cor:noComposite-noDistlaw}, if algebraic theories $\bb{S}$ and $\bb{T}$ corresponding to monads $S$ and $T$ satisfy \ref{ax:svar0}, \ref{ax:svar1}, \ref{ax:sunits}, \ref{ax:sbinary}, and \ref{ax:sidem} and \ref{ax:tvar0}, \ref{ax:tvar1}, \ref{ax:tconst}, and \ref{ax:tunit} respectively, there does not exist a distributive law $ST \Rightarrow TS$.
\end{corollary}

\begin{example}[Again: No Distributive Law for the Powerset over Itself]
 Recall from \refex~\ref{ex:STproperties} that the theory of join semilattices, corresponding to the powerset monad, satisfies all of \ref{ax:svar0}, \ref{ax:svar1}, \ref{ax:sunits}, \ref{ax:sbinary}, \ref{ax:sidem}, \ref{ax:tvar0}, \ref{ax:tvar1}, \ref{ax:tconst}, and \ref{ax:tunit}. Therefore, there is no distributive law for the powerset monad over itself. This was already shown by \citet{KlinSalamanca2018} using similar methods as in \refsec~\ref{Plotkin}. Theorem \ref{thm:nogoTreeIdem} gives a second, independent proof of this fact.
\end{example}

\begin{example}[Filling in the Gap: Multiset and Powerset Monads]
 From \citet{ManesMulry2007} we know that there are distributive laws $MM \Rightarrow MM$ and $MP \Rightarrow PM$, where $M$ is the multiset monad and $P$ the powerset monad. So the only combination of multiset and powerset that is not covered by previous theorems is $PM \Rightarrow MP$. Theorem \ref{thm:nogoTreeIdem} fills this gap, saying there is no distributive law of that type.
\end{example}

\begin{counter}[The Sweet Spot]
 We come back to the multiset monad once more. In \refcounter~\ref{counter:multisetNolist} we saw that the algebraic theory corresponding to the multiset monad had one extra equation compared to the theory for the list monad: commutativity. Because of this equation, \refpro~\ref{ax:tspecialproperty} did not hold, and therefore \refthm~\ref{thm:nogoTreeList} did not apply.

 There is a similar relation between the multiset monad and the powerset monad. Compared to the powerset monad, the theory corresponding to the multiset monad \emph{lacks} just one equation: idempotence, which is exactly what \refpro~\ref{ax:sidem} requires. The lack of this equation in the theory for the multiset monad therefore means that \refthm~\ref{thm:nogoTreeIdem} does not apply to multiset either. So multiset holds a sort of `sweet spot' in between the two no-go theorems, where a distributive law still can and does exist.
\end{counter}

%% file: comparisons.tex
\section{Algebraic Method vs Direct Verification}\label{sec:comparison}
We have now seen several proofs of no-go theorems using the algebraic method. These proofs simplify calculations and can be more insightful than the more direct approach showing the distributive law axioms cannot be satisfied. In this section we point out some of the connections between the algebraic method and the traditional proofs, and compare the two strategies.

\subsection{Translation between Categorical and Algebraic Proofs}
The two main ingredients of the algebraic proofs are substitution and the separation and essential uniqueness axioms of composite theories (see \refdef~\ref{def:composite}).
The categorical versions of the proofs use naturality and Beck's four defining axioms of distributive laws: two unit axioms and two multiplication axioms. We show how each of these correspond to specific steps in the algebraic proofs.

\begin{itemize}
\item Naturality of the distributive law embodies that variable labels are opaque, and that the choice of the set of variable labels is not significant. For the algebraic perspective, naturality implies that substituting variables for variables preserves equations.
\item The unit axioms are even simpler than naturality in the algebraic setting. They state that the variable~$x$, seen as an~$\bb{S}$-term, with $\bb{T}$-term $t$ substituted into it, is equal to the term~$t$ with all its variables seen as~$\bb{S}$-term substitutions. This simply means that:
    \[
     x[t/x] \theoryeq{U} t[x_i/x_i]
    \]
    Algebraically, this is trivial and so the unit law is handled implicitly in the algebraic proofs.
\item The multiplication axioms do not show up directly in the algebraic proofs, but they can be spotted by the careful observer. Whenever there is a substitution of $\bb{T}$-terms in $\bb{S}$-terms and all terms involved are more complicated than just variables, that would be an application of the multiplication axioms in a categorical proof. Using separation, the fact that in a composite theory all terms can be uniquely written as a $\bb{T}$-term of $\bb{S}$-terms, corresponds to following one of the paths in the multiplication axioms. A good example is the following compressed line of reasoning taken from the proof of \refthm~\ref{thm:nogoTreeList}, which clearly echoes the pattern $STT \rightarrow TST \rightarrow TTS$ seen in the second multiplication axiom:
    \begin{align*}
     & \specialopS(x_0, \specialopT(y_3,y_4))[\specialopT(y_1,y_2)/x_0] \\
    \theoryeq{U}\; & \specialopT(\specialopS(\specialopT(y_1,y_2),y_3), \specialopS(\specialopT(y_1,y_2), y_4)) \\
    \theoryeq{U}\; & \specialopT(\specialopT(\specialopS(y_1,y_3),\specialopS(y_2,y_3)),\specialopT(\specialopS(y_1,y_4),\specialopS(y_2,y_4))
    \end{align*}
\end{itemize}

Plotkin's original proof~\citep{VaraccaWinskel2006} uses the naturality of distributive laws and Beck's unit axioms, but~\emph{not} the multiplication axioms. We see this echoed in the proof of \refthm~\ref{thm:plotkin-1}, which only uses substitutions of variables for variables. The theorems in \refsec~\ref{section:beyondPlotkin} frequently use substitutions of more complicated terms, corresponding to the use of the multiplication axioms in a categorical proof. This once again shows the differences between the techniques used in the proofs of \refsec~\ref{Plotkin} and~\ref{section:beyondPlotkin}.

\subsection{Comparing Categorical and Algebraic Proofs}
In Plotkin's original proof, certain cleverly chosen terms are chased around naturality diagrams. As part of these arguments, we implicitly need to understand inverse images of the action of functions under the application of~\emph{both monads}. If we think about the monads involved as being presented by two algebraic theories, this will involve reasoning about inverse images of equivalence classes of terms in one theory with variables labelled by equivalence classes of terms in a second theory. This clarifies why generalizing this argument directly turns out to be awkward to do, and helps to explain the need to restrict to the powerset monad in previous generalizations of this approach.

The proofs in \refsec~\ref{section:beyondPlotkin} implicitly use three iterated applications of the monads involved, and so would incur even more complexity. Dealing with these proofs using three layers of equivalence classes and inverse images would become simply unmanageable. By using an explicitly algebraic approach, we can reason using equational logic, avoiding any explicit use of equivalence classes and difficulties with inverse images.

%% file: boom.tex
\section{Case Study: Boom Hierarchy}\label{section:boom}

The Boom hierarchy is a family of monadic data structures related by their corresponding algebraic theories. It contains the monads tree, list, multiset and powerset. When we look at their algebraic theories we see that they all have the same signature, consisting of a constant and a binary operation. Their axioms grow increasingly, as shown in \reftab~\ref{Boom1}.

\begin{table}[h]
  \caption{The Boom hierarchy}\label{Boom1}
  \begin{minipage}{\textwidth}
    \begin{tabular}{lcccc}
    \hline\hline
    theory \;\; & unit\;\; & associative \;\; & commutative \;\; & idempotent \\ \hline
    tree \;\; & Y  & N & N & N \\
    list \;\; & Y & Y & N & N \\
    multiset \;\; & Y & Y & Y & N \\
    powerset \;\; & Y & Y & Y & Y \\
    \hline\hline
    \end{tabular}
    \vspace{-1\baselineskip}
  \end{minipage}
\end{table}

In this section, we will consider different compositions of these monads and find out which compositions yield new monads via distributive laws. Studying the patterns of existing and non-existing distributive laws for this hierarchy is what originally inspired the search for the various no-go theorems presented in this paper. We now use the same hierarchy to demonstrate both the scope and limitations of each of the no-go theorems, and their relation to some of the positive results about distributive laws from the literature. We first look at the original Boom hierarchy. Later we expand the hierarchy to include more exotic data structures that are not used as much as those from the original hierarchy, but do give us a better overview of the scope of the various theorems about distributive laws. Similar expansions of the Boom hierarchy have been studied in \citet{UUSTALU2016}.

\begin{remark}
  The Boom hierarchy is named after the Dutchman H.J. Boom. Lambert Meertens is the first to mention this hierarchy in the literature \citep{Meertens1986}, and he attributes the idea to Boom. Also, `boom' happens to be the Dutch word for tree.
\end{remark}

\subsection{Original Boom Hierarchy}

For $S,T$ each being one of the monads tree, list, multiset or powerset, we have a complete picture of whether there exists a distributive law $ST \Rightarrow TS$, shown in \reftab~\ref{BoomDistributive1}. The positive results involve the monads multiset or powerset. These monads are both commutative, and hence they combine well with monads described by linear equations, according to \citet[Theorem 4.3.4]{ManesMulry2007}. Our \refthm~\ref{theorem} proves that these distributive laws are unique. The negative result for $PP \Rightarrow PP$ was already shown by \citet{KlinSalamanca2018} and is confirmed by both our \refthms~\ref{thm:plotkin-1} and \ref{thm:nogoTreeIdem}. The other negative results follow from either \refthm~\ref{thm:nogoTreeList} or \ref{thm:nogoTreeIdem}, and sometimes both, such as the non-existence of a distributive law $PL \Rightarrow LP$.

\begin{table}[h!]
  \caption{Distributive laws for the Boom hierarchy}\label{BoomDistributive1}
    \begin{minipage}{\textwidth}
    \vspace{1\baselineskip}
    Overview of the existence of distributive laws $[row][column] \Rightarrow [column][row]$ (Y/N), with the monads from the Boom hierarchy: tree (\treemonad), list (L), multiset (M), and powerset (P). In each case, the theorems confirming the results are indicated.

    \begin{tabular}{lllll}
    \hline\hline
     \;\; & \treemonad\;\;\;\; & L \;\;\;\; & M \;\;\;\; & P \\ \hline
    \treemonad \;\;\;\; & N\footnote{\label{ftn:boom:list} Theorem \ref{thm:nogoTreeList}} & N$^{\ref{ftn:boom:list}}$ & Y\footnote{\label{ftn:boom:MM} \citet[Theorem 4.3.4]{ManesMulry2007}} & Y$^{\ref{ftn:boom:MM}}$ \\
    L \;\;\;\; & N$^{\ref{ftn:boom:list}}$ & N$^{\ref{ftn:boom:list}}$ & Y$^{\ref{ftn:boom:MM}}$ & Y$^{\ref{ftn:boom:MM}}$ \\
    M \;\;\;\; & N$^{\ref{ftn:boom:list}}$ & N$^{\ref{ftn:boom:list}}$ & Y$^{\ref{ftn:boom:MM}}$ & Y$^{\ref{ftn:boom:MM}}$ \\
    P \;\;\;\; & N$^{\ref{ftn:boom:list}}$\footnote{\label{ftn:boom:idem} Theorem \ref{thm:nogoTreeIdem}} & N$^{\ref{ftn:boom:list}}$$^{\ref{ftn:boom:idem}}$ & N$^{\ref{ftn:boom:idem}}$  & N$^{\ref{ftn:boom:idem}}$\footnote{\label{ftn:boom:bartek} \citet[Theorem 3.2]{KlinSalamanca2018}}\footnote{\label{ftn:boom:gplotkin} Theorem \ref{thm:plotkin-1}} \\
    \hline\hline
    \end{tabular}
    \vspace{-2\baselineskip}
  \end{minipage}
\end{table}

\subsection{Extended Boom Hierarchy}

Tree, list, multiset and powerset are frequently-used data structures, which makes their grouping in the Boom hierarchy so compelling. However, it is rather restrictive if we hope to spot any patterns in the (non)-existence of distributive laws. Therefore we take a look at an extension of the Boom hierarchy, where instead of incrementally adding the properties of associativity, commutativity and idempotence, we look at all possible combinations of these properties. This yields the 8 different data structures tree, idempotent tree, commutative tree, commutative idempotent tree, list ($=$ associative tree), associative idempotent tree, multiset ($=$ associative commutative tree), and powerset ($=$ associative commutative idempotent tree). \reftab~\ref{BoomExtended} shows for which combinations of these there is a distributive law, and for which there is not, together with the theorems that prove it.

\begin{table}[h]
  \caption{Distributive laws for the extended Boom hierarchy}\label{BoomExtended}
    \begin{minipage}{\textwidth}
    \vspace{1\baselineskip}
    Overview of the existence of distributive laws $[row][column] \Rightarrow [column][row]$ (Y/N), with the monads tree (\treemonad), idempotent tree (I), commutative tree (C), commutative idempotent tree (CI), list/associative tree (L), associative idempotent tree (AI), multiset/associative commutative tree (M), and powerset/associative commutative idempotent tree (P). In each case, the theorems confirming the results are indicated.

    \begin{tabular}{lllllllll}
    \hline\hline
     \;\; & \treemonad\;\;\;\; & I \;\;\;\;\;\; & C \;\;\;\; & CI \;\;\;\; & L \;\;\;\; & AI \;\;\;\; & M \;\;\;\; & P \\ \hline
    \treemonad\;\;\;\; & N\footnote{\label{ftn:ext-boom:list} Theorem \ref{thm:nogoTreeList}} & N$^{\ref{ftn:ext-boom:list}}$ & N$^{\ref{ftn:ext-boom:list}}$ & N$^{\ref{ftn:ext-boom:list}}$ & N$^{\ref{ftn:ext-boom:list}}$ & N$^{\ref{ftn:ext-boom:list}}$ & Y\footnote{\label{ftn:ext-boom:MM} \citet[Theorem 4.3.4]{ManesMulry2007}} & Y$^{\ref{ftn:ext-boom:MM}}$ \\
    I & N$^{\ref{ftn:ext-boom:list}}$$^{\ref{ftn:ext-boom:idem}}$ & N$^{\ref{ftn:ext-boom:list}}$$^{\ref{ftn:ext-boom:idem}}$ & N$^{\ref{ftn:ext-boom:list}}$$^{\ref{ftn:ext-boom:idem}}$ & N$^{\ref{ftn:ext-boom:list}}$$^{\ref{ftn:ext-boom:idem}}$$^{\ref{ftn:ext-boom:gplotkin}}$ & N$^{\ref{ftn:ext-boom:list}}$$^{\ref{ftn:ext-boom:idem}}$ & N$^{\ref{ftn:ext-boom:list}}$$^{\ref{ftn:ext-boom:idem}}$ & N$^{\ref{ftn:ext-boom:idem}}$ & N$^{\ref{ftn:ext-boom:idem}}$$^{\ref{ftn:ext-boom:gplotkin}}$ \\
    C & N$^{\ref{ftn:ext-boom:list}}$ & N$^{\ref{ftn:ext-boom:list}}$ & N$^{\ref{ftn:ext-boom:list}}$ & N$^{\ref{ftn:ext-boom:list}}$ & N$^{\ref{ftn:ext-boom:list}}$ & N$^{\ref{ftn:ext-boom:list}}$ & Y$^{\ref{ftn:ext-boom:MM}}$& Y$^{\ref{ftn:ext-boom:MM}}$ \\
    CI & N$^{\ref{ftn:ext-boom:list}}$$^{\ref{ftn:ext-boom:idem}}$ & N$^{\ref{ftn:ext-boom:list}}$$^{\ref{ftn:ext-boom:idem}}$ & N$^{\ref{ftn:ext-boom:list}}$$^{\ref{ftn:ext-boom:idem}}$ & N$^{\ref{ftn:ext-boom:list}}$$^{\ref{ftn:ext-boom:idem}}$$^{\ref{ftn:ext-boom:gplotkin}}$ & N$^{\ref{ftn:ext-boom:list}}$$^{\ref{ftn:ext-boom:idem}}$ & N$^{\ref{ftn:ext-boom:list}}$$^{\ref{ftn:ext-boom:idem}}$ & N$^{\ref{ftn:ext-boom:idem}}$ & N$^{\ref{ftn:ext-boom:idem}}$$^{\ref{ftn:ext-boom:gplotkin}}$ \\
    L \;\; & N$^{\ref{ftn:ext-boom:list}}$ & N$^{\ref{ftn:ext-boom:list}}$ & N$^{\ref{ftn:ext-boom:list}}$ & N$^{\ref{ftn:ext-boom:list}}$ & N$^{\ref{ftn:ext-boom:list}}$ & N$^{\ref{ftn:ext-boom:list}}$ & Y$^{\ref{ftn:ext-boom:MM}}$ & Y$^{\ref{ftn:ext-boom:MM}}$ \\
    AI & N$^{\ref{ftn:ext-boom:list}}$$^{\ref{ftn:ext-boom:idem}}$ & N$^{\ref{ftn:ext-boom:list}}$$^{\ref{ftn:ext-boom:idem}}$ & N$^{\ref{ftn:ext-boom:list}}$$^{\ref{ftn:ext-boom:idem}}$ & N$^{\ref{ftn:ext-boom:list}}$$^{\ref{ftn:ext-boom:idem}}$$^{\ref{ftn:ext-boom:gplotkin}}$ & N$^{\ref{ftn:ext-boom:list}}$$^{\ref{ftn:ext-boom:idem}}$ & N$^{\ref{ftn:ext-boom:list}}$$^{\ref{ftn:ext-boom:idem}}$ & N$^{\ref{ftn:ext-boom:idem}}$ & N$^{\ref{ftn:ext-boom:idem}}$$^{\ref{ftn:ext-boom:gplotkin}}$ \\
    M \;\; & N$^{\ref{ftn:ext-boom:list}}$ & N$^{\ref{ftn:ext-boom:list}}$ & N$^{\ref{ftn:ext-boom:list}}$ & N$^{\ref{ftn:ext-boom:list}}$ & N$^{\ref{ftn:ext-boom:list}}$ & N$^{\ref{ftn:ext-boom:list}}$ & Y$^{\ref{ftn:ext-boom:MM}}$ & Y$^{\ref{ftn:ext-boom:MM}}$ \\
    P \;\; & N$^{\ref{ftn:ext-boom:list}}$\footnote{\label{ftn:ext-boom:idem} Theorem \ref{thm:nogoTreeIdem}} & N$^{\ref{ftn:ext-boom:list}}$$^{\ref{ftn:ext-boom:idem}}$ & N$^{\ref{ftn:ext-boom:list}}$$^{\ref{ftn:ext-boom:idem}}$ & N$^{\ref{ftn:ext-boom:list}}$$^{\ref{ftn:ext-boom:idem}}$$^{\ref{ftn:ext-boom:gplotkin}}$ & N$^{\ref{ftn:ext-boom:list}}$$^{\ref{ftn:ext-boom:idem}}$ & N$^{\ref{ftn:ext-boom:list}}$$^{\ref{ftn:ext-boom:idem}}$ & N$^{\ref{ftn:ext-boom:idem}}$ & N$^{\ref{ftn:ext-boom:idem}}$\footnote{\label{ftn:ext-boom:bartek} \citet[Theorem 3.2]{KlinSalamanca2018}}\footnote{\label{ftn:ext-boom:gplotkin} Theorem \ref{thm:plotkin-1}} \\
    \hline\hline
    \end{tabular}
    \vspace{-2\baselineskip}
  \end{minipage}
\end{table}

We see that it is only in the relatively special circumstances of \citet[Theorem 4.3.4]{ManesMulry2007} that distributive laws exist.

\subsection{Fully Extended Boom Hierarchy}

There is no distributive law for the list monad over itself, but as Manes and Mulry pointed out, there are at least three distributive laws for the non-empty list monad over itself. Allowing the empty structure, or in algebraic terms: a unit, clearly makes an important difference. Indeed, \refthms~\ref{thm:nogoTreeList} and \ref{thm:nogoTreeIdem} rely heavily on it. Consider now the fully extended Boom hierarchy, where every algebraic property (unit, associativity, commutativity, and idempotence) is varied to get a total of 16 different data structures. In \reftab~\ref{BoomExtendedFull}, we list all combinations of these data structures, and present our knowledge about the existence of distributive laws between them. As expected, \refthms~\ref{thm:nogoTreeList} and \ref{thm:nogoTreeIdem} do not contribute any new insights compared to \reftab~\ref{BoomExtended}. Theorem \ref{thm:plotkin-1}, however, does have more to offer.

\begin{table}[h]
  \caption{Distributive laws for the fully extended Boom hierarchy} \label{BoomExtendedFull}
    \begin{minipage}{\textwidth}
    \vspace{1\baselineskip}
    Overview of the existence of distributive laws $[row][column] \Rightarrow [column][row]$ (Y/N), with the following monads:
    tree (\treemonad), idempotent tree (I), commutative tree (C), commutative idempotent tree (CI), list/associative tree (L), associative idempotent tree (AI), multiset/associative commutative tree (M), powerset/associative commutative idempotent tree (P), and all their non-empty versions labelled by an additional +. Where neither Y or N are indicated, the existence of a distributive law is still unknown.

    \begin{tabular}{lllllllllllllllll}
    \hline\hline
     & \treemonad\;\;\;\; & I\;\;\;\;\; & C\;\;\;\; & CI\;\;\;\; & L\;\;\;\; & AI\;\;\;\; & M\;\;\;\; & P\;\;\;\; & \treemonad+\;\;\;\; & I+\;\;\;\;\; & C+\;\;\;\; & CI+\;\;\;\; & L+\;\;\;\; & AI+\;\;\;\; & M+\;\;\;\; & P+ \\ \hline
    \treemonad \;\;\;\; & N\footnote{\label{ftn:fe-boom:list} Theorem \ref{thm:nogoTreeList}} & N$^{\ref{ftn:fe-boom:list}}$ & N$^{\ref{ftn:fe-boom:list}}$ & N$^{\ref{ftn:fe-boom:list}}$ & N$^{\ref{ftn:fe-boom:list}}$ & N$^{\ref{ftn:fe-boom:list}}$ & Y\footnote{\label{ftn:fe-boom:MM} \citet[Theorem 4.3.4]{ManesMulry2007}} & Y$^{\ref{ftn:fe-boom:MM}}$ & Y$^{\ref{ftn:fe-boom:MM?}}$ & & & & & & Y$^{\ref{ftn:fe-boom:MM}}$ & Y$^{\ref{ftn:fe-boom:MM}}$\\
    I & N$^{\ref{ftn:fe-boom:list}}$$^{\ref{ftn:fe-boom:idem}}$ & N$^{\ref{ftn:fe-boom:list}}$$^{\ref{ftn:fe-boom:idem}}$ & N$^{\ref{ftn:fe-boom:list}}$$^{\ref{ftn:fe-boom:idem}}$ & N$^{\ref{ftn:fe-boom:list}}$$^{\ref{ftn:fe-boom:idem}}$$^{\ref{ftn:fe-boom:gplotkin}}$ & N$^{\ref{ftn:fe-boom:list}}$$^{\ref{ftn:fe-boom:idem}}$ & N$^{\ref{ftn:fe-boom:list}}$$^{\ref{ftn:fe-boom:idem}}$ & N$^{\ref{ftn:fe-boom:idem}}$ & N$^{\ref{ftn:fe-boom:idem}}$$^{\ref{ftn:fe-boom:gplotkin}}$ & & & & N$^{\ref{ftn:fe-boom:gplotkin}}$ & & & & N$^{\ref{ftn:fe-boom:gplotkin}}$ \\
    C & N$^{\ref{ftn:fe-boom:list}}$ & N$^{\ref{ftn:fe-boom:list}}$ & N$^{\ref{ftn:fe-boom:list}}$ & N$^{\ref{ftn:fe-boom:list}}$ & N$^{\ref{ftn:fe-boom:list}}$ & N$^{\ref{ftn:fe-boom:list}}$ & Y$^{\ref{ftn:fe-boom:MM}}$& Y$^{\ref{ftn:fe-boom:MM}}$ & Y$^{\ref{ftn:fe-boom:MM?}}$ & & & & & & Y$^{\ref{ftn:fe-boom:MM}}$ & Y$^{\ref{ftn:fe-boom:MM}}$ \\
    CI & N$^{\ref{ftn:fe-boom:list}}$$^{\ref{ftn:fe-boom:idem}}$ & N$^{\ref{ftn:fe-boom:list}}$$^{\ref{ftn:fe-boom:idem}}$ & N$^{\ref{ftn:fe-boom:list}}$$^{\ref{ftn:fe-boom:idem}}$ & N$^{\ref{ftn:fe-boom:list}}$$^{\ref{ftn:fe-boom:idem}}$$^{\ref{ftn:fe-boom:gplotkin}}$ & N$^{\ref{ftn:fe-boom:list}}$$^{\ref{ftn:fe-boom:idem}}$ & N$^{\ref{ftn:fe-boom:list}}$$^{\ref{ftn:fe-boom:idem}}$ & N$^{\ref{ftn:fe-boom:idem}}$ & N$^{\ref{ftn:fe-boom:idem}}$$^{\ref{ftn:fe-boom:gplotkin}}$ & & & & N$^{\ref{ftn:fe-boom:gplotkin}}$ & & & & N$^{\ref{ftn:fe-boom:gplotkin}}$ \\
    L\;\; & N$^{\ref{ftn:fe-boom:list}}$ & N$^{\ref{ftn:fe-boom:list}}$ & N$^{\ref{ftn:fe-boom:list}}$ & N$^{\ref{ftn:fe-boom:list}}$ & N$^{\ref{ftn:fe-boom:list}}$ & N$^{\ref{ftn:fe-boom:list}}$ & Y$^{\ref{ftn:fe-boom:MM}}$ & Y$^{\ref{ftn:fe-boom:MM}}$ & Y$^{\ref{ftn:fe-boom:MM?}}$ & & & & & & Y$^{\ref{ftn:fe-boom:MM}}$ & Y$^{\ref{ftn:fe-boom:MM}}$ \\
    AI & N$^{\ref{ftn:fe-boom:list}}$$^{\ref{ftn:fe-boom:idem}}$ & N$^{\ref{ftn:fe-boom:list}}$$^{\ref{ftn:fe-boom:idem}}$ & N$^{\ref{ftn:fe-boom:list}}$$^{\ref{ftn:fe-boom:idem}}$ & N$^{\ref{ftn:fe-boom:list}}$$^{\ref{ftn:fe-boom:idem}}$$^{\ref{ftn:fe-boom:gplotkin}}$ & N$^{\ref{ftn:fe-boom:list}}$$^{\ref{ftn:fe-boom:idem}}$ & N$^{\ref{ftn:fe-boom:list}}$$^{\ref{ftn:fe-boom:idem}}$ & N$^{\ref{ftn:fe-boom:idem}}$ & N$^{\ref{ftn:fe-boom:idem}}$$^{\ref{ftn:fe-boom:gplotkin}}$ & & & & N$^{\ref{ftn:fe-boom:gplotkin}}$ & & & & N$^{\ref{ftn:fe-boom:gplotkin}}$\\
    M \;\; & N$^{\ref{ftn:fe-boom:list}}$ & N$^{\ref{ftn:fe-boom:list}}$ & N$^{\ref{ftn:fe-boom:list}}$ & N$^{\ref{ftn:fe-boom:list}}$ & N$^{\ref{ftn:fe-boom:list}}$ & N$^{\ref{ftn:fe-boom:list}}$ & Y$^{\ref{ftn:fe-boom:MM}}$ & Y$^{\ref{ftn:fe-boom:MM}}$ & Y$^{\ref{ftn:fe-boom:MM?}}$ & & & & & & Y$^{\ref{ftn:fe-boom:MM}}$& Y$^{\ref{ftn:fe-boom:MM}}$\\
    P \;\; & N$^{\ref{ftn:fe-boom:list}}$\footnote{\label{ftn:fe-boom:idem} Theorem \ref{thm:nogoTreeIdem}} & N$^{\ref{ftn:fe-boom:list}}$$^{\ref{ftn:fe-boom:idem}}$ & N$^{\ref{ftn:fe-boom:list}}$$^{\ref{ftn:fe-boom:idem}}$ & N$^{\ref{ftn:fe-boom:list}}$$^{\ref{ftn:fe-boom:idem}}$$^{\ref{ftn:fe-boom:gplotkin}}$ & N$^{\ref{ftn:fe-boom:list}}$$^{\ref{ftn:fe-boom:idem}}$ & N$^{\ref{ftn:fe-boom:list}}$$^{\ref{ftn:fe-boom:idem}}$ & N$^{\ref{ftn:fe-boom:idem}}$ & N$^{\ref{ftn:fe-boom:idem}}$\footnote{\label{ftn:fe-boom:bartek} \citet[Theorem 3.2]{KlinSalamanca2018}}\footnote{\label{ftn:fe-boom:gplotkin} Theorem \ref{thm:plotkin-1}} & & & & N$^{\ref{ftn:fe-boom:gplotkin}}$ & & & & N$^{\ref{ftn:fe-boom:gplotkin}}$\\
    \treemonad+ \;\;     & & & & & & & Y$^{\ref{ftn:fe-boom:MM}}$& Y$^{\ref{ftn:fe-boom:MM}}$& Y\footnote{\label{ftn:fe-boom:MMtree} \citet[Example 3.9]{ManesMulry2008}} & & & & & & Y$^{\ref{ftn:fe-boom:MM}}$ & Y$^{\ref{ftn:fe-boom:MM}}$ \\
    I+ \;\;        & & & & N$^{\ref{ftn:fe-boom:gplotkin}}$ & & & & N$^{\ref{ftn:fe-boom:gplotkin}}$ & & & & N$^{\ref{ftn:fe-boom:gplotkin}}$ & & & & N$^{\ref{ftn:fe-boom:gplotkin}}$\\
    C+ \;\;        & & & & & & & Y$^{\ref{ftn:fe-boom:MM}}$& Y$^{\ref{ftn:fe-boom:MM}}$ & Y$^{\ref{ftn:fe-boom:MM?}}$ & & & & & & Y$^{\ref{ftn:fe-boom:MM}}$ & Y$^{\ref{ftn:fe-boom:MM}}$ \\
    CI+ \;\;       & & & & N$^{\ref{ftn:fe-boom:gplotkin}}$ & & & & N$^{\ref{ftn:fe-boom:gplotkin}}$ & & & & N$^{\ref{ftn:fe-boom:gplotkin}}$ & & & & N$^{\ref{ftn:fe-boom:gplotkin}}$\\
    L+ \;\;     & & & & & & & Y$^{\ref{ftn:fe-boom:MM}}$& Y$^{\ref{ftn:fe-boom:MM}}$ & Y$^{\ref{ftn:fe-boom:MM?}}$ & & & & Y\footnote{\label{ftn:fe-boom:MMne-list} \citet[Example 5.1.10]{ManesMulry2007} and \citet[Example 4.10]{ManesMulry2008}} & & Y$^{\ref{ftn:fe-boom:MM}}$ & Y$^{\ref{ftn:fe-boom:MM}}$\\
    AI+ \;\;       & & & & N$^{\ref{ftn:fe-boom:gplotkin}}$& & & & N$^{\ref{ftn:fe-boom:gplotkin}}$ & & & & N$^{\ref{ftn:fe-boom:gplotkin}}$ & & & & N$^{\ref{ftn:fe-boom:gplotkin}}$\\
    M+ \;\; & & & & & & & Y$^{\ref{ftn:fe-boom:MM}}$ & Y$^{\ref{ftn:fe-boom:MM}}$ & Y\footnote{\label{ftn:fe-boom:MM?} \citet[Example 4.9]{ManesMulry2008}} & & & & & & Y$^{\ref{ftn:fe-boom:MM}}$ & Y$^{\ref{ftn:fe-boom:MM}}$ \\
    P+ \;\; & & & & N$^{\ref{ftn:fe-boom:gplotkin}}$& & & &N$^{\ref{ftn:fe-boom:gplotkin}}$ & & & & N$^{\ref{ftn:fe-boom:gplotkin}}$ & & & & N$^{\ref{ftn:fe-boom:gplotkin}}$ \\
    \hline\hline
    \end{tabular}
    \vspace{-2\baselineskip}
  \end{minipage}
\end{table}

From this data, there is not yet a clear pattern emerging that allows us to predict whether a combination of monads will have a distributive law or not. It is our hope that by completing this table, we will learn more about the predictability or unpredictability of distributive laws between monads.

%% file: end.tex
\newpage
\section{Conclusion and Future Work}
We have shown there can be no distributive law between large classes of monads:
\begin{itemize}
\item Section~\ref{Plotkin} developed general theorems for demonstrating when distributive laws cannot exist, derived from a classical counterexample of Plotkin.
\item Section~\ref{section:beyondPlotkin} went further, showing a new approach, yielding counterexamples beyond those possible in \refsec~\ref{Plotkin}.
\item Section~\ref{section:boom} then demonstrated the praciticality of our theorems, by investigating many combinations within an extension of the Boom type hierarchy.
\end{itemize}
We strongly advocated the use of algebraic methods. These techniques were used for all of our proofs. Taking this approach, rather than direct calculations involving Beck's axioms, enabled us to single out the essentials of each proof, so that the resulting theorems could be stated in full generality.

The four no-go theorems together reveal that many combinations of monads from the extended Boom hierarchy do not compose via distributive laws. There are, however, still gaps to be filled. Since this hierarchy has inspired us in finding many of the theorems presented in this paper, completing \reftab~\ref{BoomExtendedFull} will no doubt lead to more useful insights.

Another logical direction to expand this work is generalising the results beyond the category~$\catset$.

Lastly, we would like to mention that generalising the Plotkin counterexample is not the only way for this example to be useful. Julian Salamanca has demonstrated that understanding the general idea behind the example is just as important. He has shown that there is no distributive law of the group monad over the powerset monad $GP \Rightarrow PG$ \citep{Salamanca2018}. Although the idea of his proof is very similar to the Plotkin counterexample, the proof itself uses techniques beyond our generalisations.

\section{Acknowledgements}

We are very grateful to Jeremy Gibbons, Bartek Klin and Hector Miller-Bakewell for their helpful discussions, as well as their feedback on earlier versions of our proofs and this paper. We also would like to thank Prakash Panangaden, Julian Salamanca, Fredrik Dahlqvist, Louis Parlant, Sam Staton, Ohad Kammar, Jules Hedges, and Ralf Hinze for inspiring discussions, which all had a positive influence on this paper.

This work was partially supported by Institute for Information \& communications Technology Promotion(IITP) grant funded by the Korea government(MSIT) (No.2015-0-00565, Development of Vulnerability Discovery Technologies for IoT Software Security).\\